\documentclass[12pt, reqno]{amsart}
\usepackage{amsmath, amsthm, amscd, amsfonts, amssymb, graphicx, color}
\usepackage[bookmarksnumbered, colorlinks, plainpages]{hyperref}
\hypersetup{colorlinks=true,linkcolor=red, anchorcolor=green, citecolor=cyan, urlcolor=red, filecolor=magenta, pdftoolbar=true}

\newtheorem{theorem}{Theorem}[section]
\newtheorem{lemma}[theorem]{Lemma}
\newtheorem{proposition}[theorem]{Proposition}
\newtheorem{corollary}[theorem]{Corollary}
\theoremstyle{definition}
\newtheorem{definition}[theorem]{Definition}
\newtheorem{example}[theorem]{Example}

\theoremstyle{remark}
\newtheorem{remark}[theorem]{Remark}
\numberwithin{equation}{section}

\begin{document}
	\setcounter{page}{1}
	\title[Douglas' factorization theorem and  atomic system in Hilbert pro$-C^{\ast}-$module]{Douglas' factorization theorem and  atomic system in Hilbert pro$-C^{\ast}-$module}
	
	\author[M.  Rossafi, R. Eljazzar, R. Mohapatra]{Mohamed Rossafi$^{1*}$, Roumaissae Eljazzar$^{2}$ and Ram Mohapatra$^{3}$}
	
	\address{$^{1}$LaSMA Laboratory Department of Mathematics Faculty of Sciences, Dhar El Mahraz University Sidi Mohamed Ben Abdellah, Fez, Morocco}
	\email{\textcolor[rgb]{0.00,0.00,0.84}{rossafimohamed@gmail.com}}
	
	\address{$^{2}$Laboratory of Partial Differential Equations, Spectral Algebra and Geometry, Department of Mathematics, Faculty Of Sciences, University of Ibn Tofail, Kenitra, Morocco}
	\email{\textcolor[rgb]{0.00,0.00,0.84}{roumaissae.eljazzar@uit.ac.ma}}
	
	\address{$^{3}$Department of Mathematics, University of Central Florida, Orlando, FL., 32816, USA}
	\email{\textcolor[rgb]{0.00,0.00,0.84}{ram.mohapatra@ucf.edu}}
	
	\subjclass[2010]{42C15, 46L05, 47A05.}
	
	\keywords{Douglas majorization, atomic system, Hilbert pro-$C^{\ast}$-modules.\\$^*$Corresponding author: Mohamed Rossafi (rossafimohamed@gmail.com).}

	\begin{abstract}
		In the present paper we introduce  the generalized inverse operators which have an interesting role in operator theory. We establish Douglas' factorization theorem type  for Hilbert pro-$C^{\ast}$-module. We introduce the notion of atomic system and of $K$-frame in Hilbert pro-$C^{\ast}$-module and we study the relationship between them. We also demonstrat some properties of $K$-frame by using Douglas' factorization theorem.Finally  we demonstrate that the sum of two $K$-frames in a Hilbert pro-$C^{\ast}$-module with certain conditions is once again a $K$-frame.
	\end{abstract} \maketitle
	
	\baselineskip=15pt
	
	\section*{INTRODUCTION}
	Douglas \cite{Douglas} has studied the equation $AX = B$ with a view to finding a solution for bounded linear operators on Hilbert spaces. A generalization of Douglas theorem for Hilbert $C^{\ast}$-module was given in  \cite{Fang} and \cite{Fan}.
	Those authors have extended the Douglas factor decomposition for closed
	densely defined operators on Hilbert spaces in the context of regular
	operators in Hilbert $C^{\ast}$-modules. 

	Frames, introduced by Duffin and Schaefer \cite{Duf} in 1952 to analyse some deep problems in nonharmonic Fourier series by abstracting the fundamental notion introduced by Gabor \cite{Gab} for signal processing. Today, frame theory is an exciting, dynamic and fast paced subject with applications to a wide variety of areas in mathematics and engineering, including sampling theory, operator theory, harmonic analysis, nonlinear sparse approximation, pseudodifferential operators, wavelet theory, wireless communication, data transmission with erasures, filter banks, signal processing, image processing, geophysics, quantum computing, sensor networks, and more. 
	
	In 2008, Joita  \cite{Joita} extended the theory of frames in Hilbert modules over pro-$C^{\ast}$-algebras.
	The concept of $K$-frame was first introduced by Laura G{\u{a}}vru{\c{t}}a \cite{Gavruta} in order to study atomic systems for a given bounded linear operator $K$ in a separable Hilbert space. It is well known that K-frames present a genralization  of ordinary frames, which permits to reconstruct the elements in the domain of a linear and bounded operator in a Hilbert space.  

	This paper is devoted to establishing Douglas' Factorization theorem type results for pro-$C^{\ast}$-modules. Moreover, we define the atomic system in the framework of pro-$C^{\ast}$-modules and mention some properties. Also, we present $K$-frame in Hilbert pro-$C^{\ast}$-modules and establish some results. 

	This article will be organized as follows: In section \ref{sec2},  we briefly recall
	some definitions and basic properties of pro-$C^{\ast}$-algebra. Section \ref{sec3} is devoted to introduce the concept of generalized inverse which will be used to prove a Douglas type theorem on a Hilbert  pro-$C^{\ast}$-module.
	In Section \ref{sec4} we define the atomic system and show that every Bessel sequence is an atomic system for the frame operator. Finaly we show that the sum of two $K$-frames under certain conditions is again a $K$-frame.
	\section{Preliminaries}\label{sec1}
	The basic information about pro-$C^{\ast}$-algebras can be found in the works \cite{Frag,Frago,Mallios,Inoue,Philip,Philips}.
	
	A $C^{\ast}$-algebra whose topology is induced by a family of continuous \\$C^{\ast}-$seminorms instead of a $C^{\ast}$-norm is called pro-$C^{\ast}$-algebra.  Hilbert pro-$C^{\ast}$-modules are generalizations of Hilbert spaces the inner product takes values in a pro-$C^{\ast}$-algebra rather than in the field of complex numbers.
	
	Pro-$C^{\ast}$-algebra is defined as a complete Hausdorff complex topological $\ast$-algebra $\mathcal{A}$ whose topology is determined by its continuous $C^{\ast}$-seminorms in the sens that a net $\{a_{\alpha}\}$ converges to $0$ if 	and only if $p(a_{\alpha})$ converges to $0$ for all continuous $C^{\ast}$-seminorms $p$ on $\mathcal{A}$ (see \cite{Inoue,Lance,Philips}), and we have:
	\begin{enumerate}
		\item[1)] $p(ab) \leq p(a)p(b)$
		\item[2)] $p(a^{\ast}a)=p(a)^{2}$
	\end{enumerate}
	for all  $a , b \in \mathcal{A}$.
	If the topology of pro-$C^{*}$-algebra is determined by only countably many $C^{*}$-seminorms, then it is called
	a $\sigma$-$C^{*}$-algebra.
	We denote by $sp(a)$ the spectrum of $a$ such that: $sp(a)=\left\{\lambda \in \mathbb{C}: \lambda 1_{\mathcal{A}}-a\right.$ is not invertible $\}$ for all $a \in \mathcal{A}$, where $\mathcal{A}$ is unital pro-$C^{*}$-algebra with unity $1_{\mathcal{A}}$.
	
	Let the set of all continuous $C^{\ast}$-seminorms on $\mathcal{A}$ be denoted by $S(\mathcal{A})$.
	If $\mathcal{A}^{+}$ denotes the set of all positive elements of $\mathcal{A}$, then $\mathcal{A}^{+}$ is a closed convex $C^{*}$-seminorms on $\mathcal{A}.$

	We also denote by  $\mathcal{H}_{\mathcal{A}}$  the set of all sequences ( $\left.a_{n}\right)_{n}$ with $a_{n} \in \mathcal{A}$ such that $\sum_{n} a_{n}^{*} a_{n}$ converges in $\mathcal{A}$.
	\begin{example}
		Every $C^{*}$-algebra is a pro-$C^{*}$-algebra.
	\end{example}
	\begin{proposition}\cite{Inoue}
		Let $\mathcal{A}$ be a unital pro-$C^{*}$-algebra with an identity $1_{\mathcal{A}}.$ Then for any $p \in S(\mathcal{A}),$ we have:
		\begin{enumerate}
			\item[(1)] $p(a)= p(a^{*})$ for all $a \in A$
			\item[(2)] $p \left(1_{\mathcal{A}}\right)=1$
			\item[(3)]  If $1_{\mathcal{A}} \leq b$, then $b$ is invertible and $b^{-1} \leq 1_{\mathcal{A}}$
			\item[(4)]If $a, b \in \mathcal{A}^{+}$ are invertible and $0 \leq a \leq b$, then $0 \leq b^{-1} \leq a^{-1}$
			\item[(5)] If $a, b \in \mathcal{A}^{+}$ and $a^{2} \leq b^{2}$, then $0 \leq a \leq b$
		\end{enumerate}
	\end{proposition}
	\begin{definition}\cite{Philips}
		A pre-Hilbert module over pro-$C^{\ast}$-algebra $\mathcal{A}$, is a complex vector space $E$ which is also
		a left $\mathcal{A}$-module compatible with the complex algebra structure, equipped with an $\mathcal{A}$-valued inner product $\langle .,.\rangle$ $E \times E \rightarrow \mathcal{A}$ which is $\mathbb{C}$-and $\mathcal{A}$-linear in its first variable and satisfies the following conditions:
		\begin{enumerate}
			\item[1)]
			$\langle \xi, \eta\rangle^{*}=\langle \eta, \xi\rangle 
			$ for every $\xi,\eta \in E$
			\item[2)]$\langle \xi, \xi\rangle \geq 0$ for every $\xi \in E$
			\item[3)] $\langle \xi, \xi\rangle=0 $  if and only if  $\xi=0$
		\end{enumerate}
		for every $\xi, \eta \in E .$ We say $E$ is a Hilbert $\mathcal{A}$-module (or Hilbert pro-$C^{\ast}$-module over $\mathcal{A}$ ). If $E$ is complete with respect to the topology determined by the family of seminorms
		$$
		\bar{p}_{E}(\xi)=\sqrt{p(\langle \xi, \xi\rangle)} \quad \xi \in E, p \in S(\mathcal{A})
		$$
	\end{definition}
	Let $\mathcal{A}$ be a pro-$C^{\ast}$-algebra and let $\mathcal{X}$ and $\mathcal{Y}$ be Hilbert $\mathcal{A}$-modules and assume that I and J be countable index sets.
	A bounded $\mathcal{A}$-module map from
	$\mathcal{X}$ to $\mathcal{Y}$ is called an operators from $\mathcal{X}$ to $\mathcal{Y}$. We denote the set of all operator from $\mathcal{X}$ to $\mathcal{Y}$ by  $Hom_{\mathcal{A}}(\mathcal{X}, \mathcal{Y})$.
	\begin{definition}
		An $ \mathcal{A}$-module map $T: \mathcal{X} \longrightarrow \mathcal{Y}$  is adjointable if there is a map $T^{\ast}: \mathcal{Y} \longrightarrow \mathcal{X}$ such that $\langle T \xi, \eta\rangle=\left\langle \xi, T^{\ast} \eta\right\rangle$ for all $\xi \in \mathcal{X}, \eta \in \mathcal{Y}$, and is called bounded if for all $p \in S(\mathcal{A})$, there is $M_{p}>0$ such that $\bar{p}_{\mathcal{Y}}(T \xi) \leq M_{p} \bar{p}_{\mathcal{X}}(\xi)$ for all $\xi \in \mathcal{X}$.
		
		We denote by $Hom_{\mathcal{A}}^{\ast}(\mathcal{X}, \mathcal{Y})$, the set of all adjointable operator from $\mathcal{X}$ to $\mathcal{Y}$ and $Hom_{\mathcal{A}}^{\ast}(\mathcal{X})=Hom_{\mathcal{A}}^{\ast}(\mathcal{X}, \mathcal{X})$	 
	\end{definition}
	\begin{definition}
		Let $\mathcal{A}$ be a pro-$C^{\ast}$-algebra and $\mathcal{X}, \mathcal{Y}$ be two Hilbert $\mathcal{A}$-modules. The operator $T: \mathcal{X} \rightarrow \mathcal{Y}$ is called uniformly bounded below, if there exists $C>0$ such that for each $p \in S(\mathcal{A})$,
		\begin{equation*}
			\bar{p}_{\mathcal{Y}}(T \xi) \leqslant C \bar{p}_{\mathcal{X}}(\xi), \quad \text { for all } \xi \in \mathcal{X} 
		\end{equation*}
		and  is called uniformly bounded  above if there exists $C^{\prime}>0$ such that for each $p \in S(\mathcal{A})$,
		\begin{equation*}
			\bar{p}_{\mathcal{Y}}(T \xi) \geqslant C^{\prime}  \bar{p}_{\mathcal{X}}(\xi), \quad \text { for all } \xi \in \mathcal{X}
		\end{equation*}
		\begin{equation*}
			\|T\|_{\infty}=\inf \{M: M \text { is an upper bound for } T\}
		\end{equation*}
		\begin{equation*}
			\hat{p}_{\mathcal{Y}}(T)=\sup \left\{\bar{p}_{\mathcal{Y}}(T(x)): \xi \in \mathcal{X}, \quad \bar{p}_{\mathcal{X}}(\xi) \leqslant 1\right\}
		\end{equation*}
		It's clear to see that,  $\hat{p}(T) \leqslant\|T\|_{\infty}$ for all $p \in S(\mathcal{A})$.	
	\end{definition}
	\begin{definition}\cite{Joita}
		A sequence $\left\{\xi_{i}\right\}_{i}$ in $M(\mathcal{X})$ is a standard frame of multipliers in $\mathcal{X}$ if for each $\xi \in \mathcal{X}, \sum_{i}\left\langle\xi, \xi_{i}\right\rangle_{M(\mathcal{X})}\left\langle \xi_{i}, \xi\right\rangle_{M(\mathcal{X})}$ converges in $\mathcal{A}$, and there are two positive constants $C$ and $D$ such that
		$$
		C\langle\xi, \xi\rangle_{\mathcal{X}} \leqslant \sum_{i}\left\langle\xi, \xi_{i}\right\rangle_{M(\mathcal{X})}\left\langle \xi_{i}, \xi\right\rangle_{M(\mathcal{X})} \leqslant D\langle\xi, \xi\rangle_{\mathcal{X}}
		$$
		for all $\xi \in \mathcal{X}$. If $D=C=1$ we say that $\left\{\xi_{i} \right\}_{i}$ is a standard normalized frame of multipliers.

		Particularly, if the right inequality
		$$
		\sum_{i}\left\langle\xi, \xi_{i}\right\rangle_{M(\mathcal{X})}\left\langle \xi_{i}, \xi\right\rangle_{M(\mathcal{X})} \leqslant D\langle\xi, \xi\rangle_{\mathcal{X}} \; \; \forall \xi \in \mathcal{X}
		$$
		holds true, we call $\left\{\xi_{i}\right\}_{i\in I}$ a Bessel sequence.
	\end{definition}
	\begin{definition}
		Let $\left\{\xi_{i}\right\}_{i}$ be a standard frame of multipliers in $\mathcal{X}$. The module morphism $U: \mathcal{X} \rightarrow \mathcal{H}_{\mathcal{A}}$ defined by $U(x)=$ $\left(\left\langle \xi_{i}, x \right\rangle_{M(\mathcal{X})}\right)_{i}$ is called the frame transform for $\left\{\xi_{i}\right\}_{i}$.
	\end{definition}
	\begin{definition}
		Let $\mathcal{X}$ be a Hilbert module over a pro-$C^{*}$-algebra $\mathcal{A}$ and $\left\{\xi_{i}\right\}_{i}$ be a standard frame of multipliers for $\mathcal{X}$. The invertible positive and bounded element $L$ in $Hom_{\mathcal{A}}^{\ast}(\mathcal{X})$, such that $\sum_{i} \xi_{i} \cdot\left\langle \xi_{i},L(\xi)\right\rangle_{M(\mathcal{X})}=\xi$ for all $\xi \in $ is called the frame operator associated with the standard frame of multipliers $\left\{\xi_{i}\right\}_{i}$.
	\end{definition}
	\begin{proposition}\cite{Azhini}. \label{Prop2.6}
		Let $\mathcal{X}$ be a Hilbert module over pro-$C^{*}$-algebra $\mathcal{A}$ and $T$ be an invertible element in $Hom_{\mathcal{A}}^{\ast}(\mathcal{X})$ such that both are uniformly bounded. Then for each $\xi \in \mathcal{X}$,
		$$
		\left\|T^{-1}\right\|_{\infty}^{-2}\langle \xi, \xi\rangle \leq\langle T \xi, T \xi \rangle \leq\|T\|_{\infty}^{2}\langle \xi, \xi\rangle
		.$$
	\end{proposition}
	\begin{theorem}\cite{Joitaa}\label{2.10} 
		Let $T \in Hom_{\mathcal{A}}^{\ast}(\mathcal{X},\mathcal{Y})$. If $T$ has closed range then:
		\begin{enumerate}
			\item $Ker (T)$ is a complemented submodule of $\mathcal{X}$;
			\item  $Ran(T)$, the range of $T$, is a complemented submodule of $\mathcal{Y}$.
		\end{enumerate}
	\end{theorem}
	\begin{proposition}\cite{Zhura}\label{prop2.11}
		Let $T: \mathcal{X} \rightarrow \mathcal{X}$ be an operator, then the following statements are equivalent:
		\begin{enumerate}
			\item $T$ is a positive element in $Hom_{\mathcal{A}}^{\ast}(\mathcal{X})$ 
			\item  for any element $\xi \in \mathcal{X}$ the inequality $\langle T \xi, \xi\rangle\geq 0$ holds, i.e. this element is positive in
			$\mathcal{A}$. \end{enumerate}
	\end{proposition}
	\begin{lemma}\cite{Inoue}\label{lem2.11}
		Let $\mathcal{A}$ be a pro-$C^{*}$-algebra, suppose $\alpha, \beta \in \mathcal{A}^{+}$be such that $\alpha \leq \beta$. Then $p(\alpha) \leq p(\beta)$ and $\lambda^{*} \alpha \lambda \leq \lambda^{*} \beta \lambda$ for each $\lambda \in \mathcal{A}$.
	\end{lemma}
	\section{Douglas' factorization theorem in Hilbert pro$-C^{\ast}-$module}\label{sec2}
	We start by defining the bounded generalized inverse module map. We give an equivalente characterization of a bounded generalized inverse $\mathcal{A}$-module map, which is the main tool to obtain Douglas' type factor decomposition theorem of some uniformly bounded module maps.
	
	\begin{definition}\label{3.1}
		Let $\mathcal{X}$ and $\mathcal{Y}$ be two Hilbert pro-$C^{*}$-modules over a pro-$C^{*}$-algebra $\mathcal{A}$, $T \in Hom_{\mathcal{A}}^{\ast}(\mathcal{X},\mathcal{Y})$. If there exists a $T^{\dagger} \in Hom_{\mathcal{A}}^{\ast}(\mathcal{X},\mathcal{Y})$ such that
		\begin{enumerate}
			\item[(1)] $T T^{\dagger} T=T$;
			\item[(2)] $T^{\dagger} T T^{\dagger}=T^{\dagger}$; 
			\item[(3)] $\left(T T^{\dagger}\right)^{*}=T T^{\dagger} ;$
			\item[(4)] $\left(T^{\dagger} T\right)^{*}=T^{\dagger} T$,
		\end{enumerate}
		then $T^{\dagger}$ is called a bounded generalized inverse module map of $T$.
	\end{definition} 
	\begin{remark}\label{rmk2.2}
		According to the conditions of the previous definition, we can easily deduce that $T^{\dagger}$ is unique. In fact suppose that $T$ has another generalized inverse which we denote by $T^{\prime}$. The definition \ref{3.1} implies  $Ran\left(T T^{\dagger}\right)=Ran(T)=Ran\left(T T^{\prime}\right)$ and $Ran\left(T^{\dagger} T\right)=Ran\left(T^{*}\right)=Ran\left(T^{\prime} T\right)$. Note that $T T^{\dagger}, T^{\dagger} T, T T^{\prime}, T^{\prime} T$ are all projections, hence $T T^{\dagger}=T T^{\prime}$ and $T^{\dagger} T=T^{\prime} T$. As a consequence,, $T^{\dagger}=T^{\dagger} T T^{\dagger}=T^{\dagger} T T^{\prime}=T^{\prime} T T^{\prime}=T^{\prime}$.
	\end{remark}
	\begin{theorem}\label{3.2}
		Let $T \in Hom_{\mathcal{A}}^{\ast}(\mathcal{X},\mathcal{Y})$, and let $\mathcal{X}, \mathcal{Y}$ be Hilbert pro-$C^{\ast}$-modules over a pro-$C^{*}$-algebra $\mathcal{A}$. Then the following statements are equivalent:
		\begin{enumerate}
			\item[(1)] $T$ has a generalized inverse module map in $Hom_{\mathcal{A}}^{\ast}(\mathcal{X},\mathcal{Y})$,
			\item[(2)] $Ran(T)$ is a closed submodule in $\mathcal{Y}$.
		\end{enumerate}
	\end{theorem}
	\begin{proof}
		$(1) \Rightarrow(2)$ is easy to prove, see the Remark \ref{rmk2.2}.
		
		$(2) \Rightarrow(1)$. Let $Ran(T)$ be closed. By Theorem \ref{2.10}, $Ker(T)$ and $Ran(T)$ are both complemented submodules in $\mathcal{X}$ and in $\mathcal{Y}$, respectively, then $\mathcal{X}=Ker(T) \oplus Ran\left(T^{*}\right)$ and $\mathcal{Y}=Ker\left(T^{*}\right) \oplus Ran(T)$. Define a map $T^{\dagger}: \mathcal{Y} \mapsto \mathcal{X}$ which is linear by
		$$
		T^{\dagger} \xi= \begin{cases}\left(T \mid Ker(T)^{\perp}\right)^{-1} \xi & \xi \in Ran(T) \\ 0 & \xi \in Ker\left(T^{*}\right)\end{cases}
		$$
		and  define a linear map $\left(T^{\dagger}\right)^{*}: \mathcal{X} \mapsto \mathcal{Y}$ by
		$$
		\left(T^{\dagger}\right)^{*} \xi= \begin{cases}\left(T^{*} \mid Ker\left(T^{*}\right)^{\perp}\right)^{-1} \xi & \xi \in Ran\left(T^{*}\right) \\ 0 & \xi \in Ker(T) .\end{cases}
		$$
		To prove that $T^{\dagger}$is the generalized inverse module map of $T$, we should prove that  $T^{\dagger} \in Hom_{\mathcal{A}}^{\ast}(\mathcal{X},\mathcal{Y})$, equivalently, $\left\langle T^{\dagger} \xi, \eta\right\rangle=\left\langle \xi,\left(T^{\dagger}\right)^{*} \eta\right\rangle, \xi \in \mathcal{Y}$, $\eta \in \mathcal{X}$. The verification of this identity is uncomplicated  utilizing the orthogonal direct sum decompositions. It's also easy to verify the conditions $(1)-(4)$ of Definition \ref{3.1}
	\end{proof}

	\begin{theorem}\label{3.3}
		Let $\mathcal{X}$ be a Hilbert $\mathcal{A}$-module over a pro-$C^{*}$-algebra $\mathcal{A} .$ Let $T, L \in Hom_{\mathcal{A}}^{*}(\mathcal{X})$.
		If  $Ran(L)$ is closed, then the following statements are equivalent:
		\begin{enumerate}
			\item $Ran(T) \subseteq Ran(L)$.
			\item  $T T^{*} \leq \alpha^{2} L L^{*}$ for some $\alpha \geq 0$.
			\item  There exists $U \in Hom_{\mathcal{A}}^{*}(\mathcal{X})$ uniformly bounded such that $T=L U$.
		\end{enumerate}
	\end{theorem}
	\begin{proof} 
		$(3) \Longrightarrow (1)$ is obvious.
		
		$(1) \Longrightarrow (3)$ Suppose that $(1)$ holds. For every $\xi \in \mathcal{X},$ we have $T(\xi) \in Ran(L)$. Since $L$ has closed range, then by Theorem \ref{2.10} $Ker(L)$ is a complemented submodule, which it results that there exists a unique $\eta \in KerL^{\perp}$ such tat $T(\xi)=L(\eta)$. Let's define the map $U$ as follow $U: U(\xi)=\eta$. Observe the fact that then $U(\xi a)=\eta a$ indicates that $U$ is a module map, for each $a \in \mathcal{A}$. The construction of $U$ gives $T=L U$. We will next show  that $U \in Hom_{\mathcal{A}}^{*}(\mathcal{X})$. By Theorem \ref{3.2} $L$ has the generalized inverse module map $ L^{\dagger} \in Hom_{\mathcal{A}}^{*}(\mathcal{X})$. Since $T=L U$ then $L^{\dagger} T=L^{\dagger} L U$. Note that $ker L^{\perp}=Ran \left(L^{*}\right)$, hence $L^{\dagger} L U(\xi)=L^{\dagger} L(\eta)=\eta=U(\xi)$, for every $\xi \in \mathcal{X}$, that is, $L^{\dagger} L U=U$. According to the above proof we have $U=L^{\dagger} T \in Hom_{\mathcal{A}}^{*}(\mathcal{X})$. In fact, $U^{*}=T^{*}\left(L^{\dagger
		}\right)^{*}$.
		
		$(3)\Longrightarrow (2)$ Suppose that (3) holds. then there exists  $U \in Hom_{\mathcal{A}}^{*}(\mathcal{X})$ such that $T=L U$, $T T^{*}=L U U^{*} U^{*} \leq\|U\|_{\infty}^{2} L L^{*}$, so $T T^{*} \leq \lambda^{2} L L^{*}$, where $\lambda=\|U\|_{\infty}$.
		
		$(2) \Longrightarrow (3)$ Let be $U_{1}$  a module map defined as follow $U_{1}: L^{*} \mathcal{X} \longrightarrow T^{*} \mathcal{X}$ such that $U_{1} L^{*}(\xi)=T^{*}(\xi), \forall \xi \in \mathcal{X}.$ By $(2)$ we have that $\bar{p}_{\mathcal{X}}(U_{1}\left(L^{*} \xi\right))^{2}=\bar{p}_{\mathcal{X}}(T^{*})^{2}=p(\sqrt{\left\langle T T^{*} \xi, \xi\right\rangle}) \leq p(\sqrt{\lambda^{2}\left\langle L L^{*} \xi, \xi\right\rangle})=\lambda^{2}\bar{p}_{\mathcal{X}}(L^{*} \xi)^{2}$. As well as $U_{1}$ is well defined and uniformly bounded. Since $R(L)$ is closed, and also $R\left(L^{*}\right)$ \cite[Theorem 3.2.4]{Joitaa}. Out of the observation above, setting $U_{1} \xi=0, \xi \in\left(L^{*} \mathcal{X}\right)^{\perp}$, it is simple to verify that $U_{1}$ is bounded in $\mathcal{X}$, and from the construction of $U_{1}$ we have $U_{1} L^{*}=T^{*}$. Notify that $L^{*}\left(L^{*}\right)^{\dagger} L^{*}=L^{*}$, so $U_{1} L^{*}(\xi)=$ $U_{1} L^{*}\left(\left(L^{*}\right)^{\dagger} L^{*}(\xi)\right)=T^{*}\left(\left(L^{*}\right)^{\dagger} L^{*}(\xi)\right)$, from which it results that $T^{*}=T^{*}\left(L^{*}\right)^{\dagger} L^{*}$, i.e, $T=L L^{\dagger} T$ (in this case by using the fact $\left(L^{\dagger}\right)^{*}=\left(L^{*}\right)^{\dagger}$), from which $T=L U$, where $U=L^{\dagger} T$.

	\end{proof}

	\begin{theorem}\label{3.4}
		Let  $T, L \in Hom_{\mathcal{A}}^{*}(\mathcal{X})$ be two uniformly bounded operators, then
		$$
		Ran(T)+Ran(L)=Ran(\sqrt {T T^{\ast}+L L^{*}}) .
		$$
	\end{theorem}
	\begin{proof}
		Let $M= \left(\begin{array}{cc}
			T &-L \\
			0 & 0
		\end{array}\right)$ on $\mathcal{X}\oplus \mathcal{X}$. Then we have 
		\begin{align*}
			( Ran(T)+Ran(L) ) \oplus \{0\}&=Ran(M)\\&=Ran(M M^{\ast})^{1 / 2}\\&=Ran\left(\begin{array}{cc}
				\left(TT^{*}+L L^{*}\right)^{1 / 2} & 0 \\
				0 & 0
			\end{array}\right)
			\\&=Ran\left(\sqrt{T T^{*}+LL^{*}}\right) \oplus\{0\} .
		\end{align*}
		Which gives $$
		Ran(T)+Ran(L)=Ran(\sqrt {T T^{\ast}+L L^{*}}) .
		$$

	\end{proof}
	\begin{corollary} \label{crll3.5}
		Let $T, L_{1}, L_{2} \in Hom_{\mathcal{A}}^{\ast}(\mathcal{X})$. The statements below are equivalent:
		\begin{enumerate}
			\item[(1)] $Ran(T) \subset Ran\left(L_{1}\right)+Ran \left(L_{2}\right)$.
			\item[(2)] $T T^{*} \leqslant \alpha^{2}\left(L_{1} L_{1}^{\ast}+L_{2} L_{2}^{*}\right)$ for some constant $\alpha>0$.
			\item[(3)] There are two uniformly bounded operators $U$ and $V$ such that $T=L_{1} U+L_{2} V$.
		\end{enumerate}
	\end{corollary}
	\begin{proof}
		By using Theorems \ref{3.2} and \ref{3.4} we observe that the conditions $(1)$ and $(2)$ are equivalent. 
		
		Since it is evident that the condition $(3)$ implies $(1)$, it is enough to prove that the statement $(1)$ implies $(3)$.
		
		Now if we assume that $M=\left(\begin{array}{ll}
			T & 0 \\
			0 & 0
		\end{array}\right) $ and $ N=\left(\begin{array}{ll}
			L_{1} & L_{2} \\
			0 & 0
		\end{array}\right)$, the assertion $(1)$ implies $Ran(M) \subset Ran(N)$. Therefore $M=NU$ for some $2\times2$ matrix $U$ and then Theorem \ref{3.2} implies $(3)$.

	\end{proof}
	
	\begin{lemma}\label{lem3.7}
		Let $\mathcal{A}$ be a pro-$C^{*}$-algebra and $\alpha, \beta \in \mathcal{A}^{+}$be such that
		$$
		p(\alpha \lambda) \leq p(\beta \lambda) \text { for all } \lambda \in \mathcal{A}^{+} .
		$$
		Then $\alpha^{2} \leq \beta^{2}$.
	\end{lemma}
	\begin{proof}
		Without loss of generality, assume that $p(\alpha) \leq 1$ and $p(\beta) \leq 1$ . Suppose that the inequality $\alpha^{2} \leq \beta^{2}$ is not true. \\Then there exists $x_{0} \in \operatorname{sp}\left(\alpha^{2}-\beta^{2}\right)$ such that $x_{0}>0$, where $\operatorname{sp}\left(\alpha^{2}-\beta^{2}\right)$ denotes the spectrum of $\alpha^{2}-\beta^{2}$. It follows that
		$$
		m=\max \left\{x: x \in \operatorname{sp}\left(\alpha^{2}-\beta^{2}\right)\right\}>0
		$$
		Let $f$ be any continuous real-valued function defined on the real line such that
		$$
		\left\{
		\begin{array}{ll}
			&0 \leq f(t) \leq 1 \text { for all } t \in \mathbb{R} \\
			&f(t)=0 \text { for all } t \in\left(-\infty, \frac{m}{2}\right] \\
			&f(t)=1 \text { for all } t \in[m,+\infty)
		\end{array}
		\right.
		$$
		Choose $\lambda=f\left(\alpha^{2}-\beta^{2}\right) .$ Then $\lambda \in \mathcal{A}^{+}$since $f$ is non-negative. By use of the functional calculus, we have
		\begin{equation}\label{3.5}
			p(\lambda\left(\alpha^{2}-\beta^{2}\right) \lambda)=\max \left\{x f(x)^{2}: x \in \operatorname{sp}\left(\alpha^{2}-\beta^{2}\right)\right\}=m>0
		\end{equation}
		Since $x>\frac{m}{2}$ whenever $f(x) \neq 0$, it follows again by use of the functional calculus that
		\begin{equation}\label{eq3.4}
			\lambda\left(\alpha^{2}-\beta^{2}\right) \lambda \geq \frac{m}{2} \lambda^{2} .
		\end{equation}
		Note that $\lambda \beta^{2} \lambda$ is self-adjoint, so there exists a state $\rho$ acting on $\mathcal{A}$ such that $\phi\left(\lambda \beta^{2} \lambda\right)=p(\lambda \beta^{2} \lambda)$. Then by \ref{eq3.4} and the assumption $p(\beta) \leq 1$, we have
		\begin{align*}
			p(\beta \alpha^{2} \beta) \geq \phi\left(\lambda \alpha^{2} \lambda\right) &\geq \phi\left(\lambda \beta^{2} \lambda\right)+\frac{m}{2} \phi \left(\lambda^{2}\right) \\&\geq \frac{m+2}{2} \phi\left(\lambda \beta^{2} \lambda\right)=\frac{m+2}{2}p(\lambda \beta^{2} \lambda)
		\end{align*}
		This shows that if $\lambda \beta^{2} \lambda \neq 0$, then $p(\lambda \alpha^{2} \lambda)>p(\lambda \beta^{2} \lambda) .$ On the other hand, if $\lambda \beta^{2} \lambda=0$, then it follows from \ref{3.5} that $p(\lambda \alpha^{2} \lambda)>0$. So in either case, we have
		$$
		p(\alpha \lambda)^{2}=p(\lambda \alpha^{2} \lambda)>p(c b^{2} c)=p(\beta \lambda)^{2},
		$$
		in contradiction to the assumption that $p(\alpha \lambda) \leq p(\beta \lambda)$.

	\end{proof}
	\begin{theorem}\label{th3.8}
		Let $\mathcal{X}, \mathcal{Y}$ be two Hilbert $\mathcal{A}$-modules, and let $T \in Hom_{\mathcal{A}}^{\ast}(\mathcal{X},\mathcal{Y})$ and $L \in  Hom_{\mathcal{A}}^{\ast}(\mathcal{X},\mathcal{Y}).$ Then the following two statements are equivalent:
		\begin{enumerate}
			\item[(i)] $T T^{*} \leq L L^{*}$;
			\item[(ii)] $\bar{p}_{\mathcal{X}}(T^{*} \xi) \leq \bar{p}_{\mathcal{X}}(L^{*} \xi)$ for all $\xi \in \mathcal{Y}$.
		\end{enumerate}
	\end{theorem}
	\begin{proof}
		$(i)\Rightarrow (ii)$ follows from Proposition \ref{prop2.11} and Lemma \ref{lem2.11}.
		
		$(ii)\Rightarrow (i)$  let $\xi \in \mathcal{Y}$ and $\alpha=\langle T T^{\ast} \xi, \xi \rangle$ and $\beta=\langle L L^{\ast} \xi, \xi \rangle$. Then $\alpha , \beta \in {\mathcal{A}}^{+}$, for any $\lambda \in \mathcal{A}^{+}$ it is established that
		
		\begin{align*}
			p(\alpha^{\frac{1}{2}} \lambda)^{2}&=p((\alpha^{\frac{1}{2}} \lambda)^{\ast}\alpha^{\frac{1}{2}}\lambda)
			=p(\lambda^{\ast}\alpha \lambda)=p(\lambda\alpha \lambda)=p(\lambda\langle T T^{*} \xi, \xi\rangle \lambda)\\&=p(\langle T T^{*}(\xi \lambda),(\xi \lambda)\rangle)
			\leq p(\langle L L^{*}(\xi \lambda),(\xi \lambda)\rangle)=p(\beta^{\frac{1}{2}} \lambda)^{2}
		\end{align*}
		
		As a result, by Lemma \ref{lem3.7}, we have $\alpha \leq \beta$ and so
		$$
		\left\langle T T^{*} \xi, \xi\right\rangle \leq\left\langle LL^{*} \xi, \xi\right\rangle
		$$
		Therefore  by Proposition \ref{prop2.11}, that $T T^{*} \leq L L^{*}$.

	\end{proof}
	\section{Atomic system in $Hom_{\mathcal{A}}^{\ast}(\mathcal{X})$}\label{sec3}
	Next, we introduce the concept of atomic systems for operators in Hilbert pro-$C^{\ast}$-modules and we establish some results.
	\begin{definition}
		Let $K \in Hom_{\mathcal{A}}^{\ast}(\mathcal{X})$, we say that $\{\xi_{i}\}_{i=1}^{\infty}$ is an atomic system for $K$ if :
		\begin{enumerate}
			\item[i)] The serie $\sum_{i} m_{i} \xi_{i}$ converges for all $m=\left(m_{i}\right) \in \mathcal{H}_{\mathcal{A}}$;
			\item[ii)]  There exists $C>0$ such that for every $x \in \mathcal{X}$ there exists $m_{x}=\left(m_{i}\right) \in \mathcal{H}_{\mathcal{A}}$ such that $\langle m_{x},m_{x}\rangle_{\mathcal{H}_{\mathcal{A}}} \leqslant C \langle x, x \rangle_{\mathcal{X}}$ and $K x=\sum_{i} m_{i} \xi_{i}$.
		\end{enumerate}
	\end{definition}
	\begin{proposition}
		Suppose that  $\left\{\xi_{i}\right\}_{i=1}^{\infty}$ is a Bessel sequence in $\mathcal{X}$. Then $\left\{\xi_{i}\right\}_{i=1}^{\infty}$ is an atomic system for  $L$, where $L$ is the frame operator.
	\end{proposition}
	\begin{proof}
		We have the module morphism 
		\begin{center}
			$\theta : \mathcal{X} \rightarrow {\mathcal{H}_{\mathcal{A}}} $ defined by $\theta x= \{\left\langle x, \xi_{i}\right\rangle\}_{i=1}^{\infty}$,
		\end{center}
		and its adjoint
		\begin{center}
			$\theta : {\mathcal{H}_{\mathcal{A}}}  \rightarrow \mathcal{X}$ defined by $\theta \left(\left\{c_{i}\right\}_{i=1}^{\infty}\right)=\sum_{i=1}^{\infty} c_{i} \xi_{i}$.
		\end{center}
		Let $L=\theta^{\ast}\theta$. Then 
		\begin{center}
			$L: \mathcal{X} \rightarrow \mathcal{X}, \quad L x=\sum_{i=1}^{\infty}\left\langle x, \xi_{i}\right\rangle \xi_{i}$
		\end{center}
		Let $a_{\xi}=\{a_{i}\}=\{ \langle \xi , \xi_{i}\rangle\}_{i=1}^{\infty} \in{\mathcal{H}_{\mathcal{A}}} $
		
		Now
		$$\begin{aligned}
			\bar{p}_{\mathcal{X}}(a_{\xi})^{2}=&p(\langle a_{\xi} , a_{\xi}\rangle)\\
			=&p\left(\sum_{i \in I}\left\langle \xi, \xi_{i}\right\rangle_{M(\mathcal{X})}\left\langle \xi_{i}, \xi\right\rangle_{M(\mathcal{X})}\right) \\
			\leq& D		\bar{p}_{\mathcal{X}}(\xi)^{2}
		\end{aligned}$$
		Then, $\left\{\xi_{i}\right\}_{i=1}^{\infty}$ is an atomic system for $L$.

	\end{proof}
	\begin{lemma}\label{lem3.3}
		Let $\mathcal{A}$ be  a unital pro-$C^{*}$-algebra   and  $\left\{\xi_{i}\right\}_{j \in I}$ be a sequence of a finitely or countably generated Hilbert $\mathcal{A}$-module $\mathcal{X}$ over $\mathcal{A}$. Then $\left\{\xi_{i}\right\}_{i \in I}$ is a Bessel sequence with bound $D$ if and only if
		$$
		p(\sum_{i \in I}\left\langle x, \xi_{i}\right\rangle\left\langle \xi_{i}, x\right\rangle) \leq D\bar{p}_{\mathcal{X}}(\xi)^{2}
		$$ for all $\xi \in \mathcal{X}.$
	\end{lemma}
	\begin{proof}
		$"\Rightarrow"$ Evident
		
		$"\Leftarrow"$
		Let $T$ be  a linear operator defined as follow
		$T: \mathcal{X} \rightarrow \mathcal{H}_{\mathcal{A}}$ 
		$$
		T \xi=\sum_{i \in I}\left\langle \xi, \xi_{i}\right\rangle e_{i}, \quad \forall \xi \in \mathcal{X} .
		$$
		Then
		$$
		\bar{p}_{\mathcal{X}}(T\xi)^{2}
		=p(\langle T \xi, T \xi\rangle)=p(\sum_{i \in I}\left\langle \xi, \xi_{i}\right\rangle\left\langle \xi_{i}, \xi\right\rangle) \leq D\bar{p}_{\mathcal{X}}(\xi)^{2},
		$$g
		which implies that  $T$ is bounded.

		It is clear that $T$ is $\mathcal{A}$-linear. Then we have
		$$
		\langle T \xi, T \xi \rangle \leq D\langle \xi, \xi \rangle
		$$
		equivalently, $\sum_{i \in I }\left\langle \xi, \xi_{i}\right\rangle\left\langle \xi_{i}, \xi \right\rangle \leq D\langle \xi, \xi\rangle$, as desired.

	\end{proof}
	\begin{theorem}
		If $K \in Hom_{\mathcal{A}}^{\ast}(\mathcal{X})$, then there exists an atomic system for the operator $K$.
	\end{theorem}
	\begin{proof}
		Let $\left\{\xi_{i}\right\}_{i \in I}$ be a standard normalized tight frame for $\mathcal{X}$. Since
		$$
		\xi=\sum_{i \in I}\left\langle \xi, \xi_{i}\right\rangle \xi_{i}, \quad \xi \in \mathcal{X},
		$$
		we have
		$$
		K \xi=\sum_{i \in I}\left\langle \xi, \xi_{i}\right\rangle K \xi_{i}, \quad \xi \in \mathcal{X} .
		$$
		For $\xi \in \mathcal{H}$, putting $a_{i, \xi}=\left\langle \xi, \xi_{i}\right\rangle$ and $f_{i}=K \xi_{i}$ for all $i \in I$, we get
		$$
		\begin{aligned}
			\sum_{i \in I}\left\langle \xi, f_{i}\right\rangle\left\langle f_{i}, \xi\right\rangle &=\sum_{i \in I}\left\langle \xi, K \xi_{i}\right\rangle\left\langle K \xi_{i}, \xi\right\rangle \\
			&=\sum_{i \in I}\left\langle K^{*} \xi, \xi_{i}\right\rangle\left\langle \xi_{i}, K^{*} \xi\right\rangle=\left\langle K^{*} \xi, K^{*} \xi\right\rangle \\
			& \leqslant\left\|K^{*}\right\|_{\infty}^{2}\langle \xi, \xi\rangle .
		\end{aligned}
		$$
		Therefore $\left\{f_{i}\right\}_{i \in I}$ is a Bessel sequence for $\mathcal{H}$ and we conclude that the series $\sum_{i \in I} c_{i} f_{i}$ converges for all $c=\left\{c_{i}\right\}_{i \in I} \in \ell^{2}(A)$ by Lemma \ref{lem3.3}. We also have
		$$
		\sum_{i \in I} a_{i, \xi} a_{i, \xi}^{*}=\sum_{i \in I}\left\langle \xi, \xi_{i}\right\rangle\left\langle \xi_{i}, \xi\right\rangle=\langle \xi, \xi\rangle,
		$$
		which completes the proof.

	\end{proof}
Our next result deals with Bessel sequences.
	\begin{theorem}\label{thm3.5}

		Let $\left\{\xi_{i}\right\}_{i \in I}$ be a Bessel sequence for $\mathcal{X}$ and $K \in Hom_{\mathcal{A}}^{\ast}(\mathcal{X})$. Suppose that $T \in Hom_{\mathcal{A}}^{\ast}\left(\mathcal{X}, \mathcal{H}_{\mathcal{A}}\right)$ is given by $T(\xi)=\left\{\left\langle \xi, \xi_{i}\right\rangle\right\}_{i \in I}$ and $\overline{Ran(T)}$ is orthogonally complemented. Then the following statements are equivalent:
		\begin{enumerate}
			\item $\left\{\xi_{i}\right\}_{i \in I}$ is an atomic system for $K$;
			\item there exists two positives values $C$ and $D$ such that 
			$$
			\begin{aligned}
				&C^{-1}\left(\bar{p}_{\mathcal{X}}\left(K^{*} \xi\right)\right)^{2} \leq p \left(\sum_{i \in l}\left\langle \xi, \xi_{i}\right\rangle_{M(\mathcal{X})}\left\langle \xi_{i}, \xi\right\rangle_{M(\mathcal{X})}\right) 
				&\leq D\left(\bar{p}_{\mathcal{X}}(\xi)\right)^{2} ;
			\end{aligned}
			$$ 
			\item  There exists $Q \in  Hom_{\mathcal{A}}^{\ast}\left(\mathcal{X}, \mathcal{H}_{\mathcal{A}}\right)$ such that $K=T^{*} Q$.
			
		\end{enumerate}
		
	\end{theorem}
	\begin{proof}
		$(1) \Rightarrow (2)$
		For each $\xi \in \mathcal{X}$
		$$
		\left(\bar{p}_{\mathcal{X}}\left(K^{*} \xi\right)\right)^{2}=\left(\sup \left\{p(\langle\xi, K y\rangle): \bar{p}_{\mathcal{X}}(y) \leq 1\right\}\right)^{2}
		$$
		And 
		$K y=\sum_{i} m_{i} \xi_{i}$.
		Then 
		$$
		\begin{aligned}
			\left(\bar{p}_{\mathcal{X}}\left(K^{*} \xi\right)\right)^{2}&=\left(\sup \left\{p\left(\left\langle \xi, \sum_{i \in l} \xi_{i} m_{i}\right\rangle\right): \bar{p}_{\mathcal{X}}(y) \leq 1\right\}\right)^{2} \\
			&=\left(\sup \left\{p\left(\sum_{i \in I} m_{i}^{*}\left\langle \xi_{i}, \xi\right\rangle_{M(\mathcal{X})}\right): \bar{p}_{\mathcal{X}}(y) \leq 1\right\}\right)^{2} \\
			&\leq \sup _{\bar{p}_{\mathcal{X}}(y) \leq 1}\left\{p\left(\sum_{i \in I} m_{i}^{*} m_{i}\right) p\left(\sum_{i \in l}\left\langle \xi, \xi_{i}\right\rangle_{M(\mathcal{X})}\left\langle \xi_{i}, \xi\right\rangle_{M(\mathcal{X})}\right)\right\} \\
			&\leq C p\left(\sum_{i \in I}\left\langle \xi, \xi_{i}\right\rangle_{M(\mathcal{X})}\left\langle \xi_{i}, \xi\right\rangle_{M(\mathcal{X})}\right)
		\end{aligned}
		$$
		Then
		$$
		C^{-1}\left(\bar{p}_{\mathcal{X}}\left(K^{*} \xi\right)\right)^{2} \leq p\left(\sum_{i \in l}\left\langle \xi, \xi_{i}\right\rangle_{M(\mathcal{X})}\left\langle \xi_{i}, \xi\right\rangle_{M(\mathcal{X})}\right)
		$$
		Since $\left\{\xi_{i}\right\}_{i=1}^{\infty}$ is Bessel sequence, then there exists positive value $D$ such that
		$$	p \left(\sum_{i \in l}\left\langle \xi, \xi_{i}\right\rangle_{M(\mathcal{X})}\left\langle \xi_{i}, \xi\right\rangle_{M(\mathcal{X})}\right) 
		\leq D\left(\bar{p}_{\mathcal{X}}(\xi)\right)^{2}$$
		Then $$
		\begin{aligned}
			&C^{-1}\left(\bar{p}_{\mathcal{X}}\left(K^{*} \xi\right)\right)^{2} \leq p \left(\sum_{i \in l}\left\langle \xi, \xi_{i}\right\rangle_{M(\mathcal{X})}\left\langle \xi_{i}, \xi\right\rangle_{M(\mathcal{X})}\right) 
			&\leq D\left(\bar{p}_{\mathcal{X}}(\xi)\right)^{2} .
		\end{aligned}
		$$
		
		$(2) \Rightarrow(3)$ Since $\left\{\xi_{i}\right\}_{i \in I}$ is a Bessel sequence, we get $T^{*} e_{i}=\xi_{i}$, where $\left\{e_{i}\right\}_{i \in I}$ is the standard orthonormal basis for $\mathcal{H}_{\mathcal{A}} $. Therefore
		$$
		\begin{aligned}
			C^{-1}\left(\bar{p}_{\mathcal{X}}\left(K^{*} \xi\right)\right)^{2} \leq p \left(\sum_{i \in l}\left\langle \xi, \xi_{i}\right\rangle_{M(\mathcal{X})}\left\langle \xi_{i}, \xi\right\rangle_{M(\mathcal{X})}\right)&=p\left(\sum_{i \in I}\left\langle \xi, T^{*} e_{i}\right\rangle\left\langle T^{*} e_{i}, \xi\right\rangle\right) \\
			&=p\left(\sum_{i \in I}\left\langle T \xi, e_{i}\right\rangle\left\langle e_{i}, T \xi\right\rangle\right)=\bar{p}_{\mathcal{X}}(T \xi)^{2}, \quad \xi\in \mathcal{X}
		\end{aligned}
		$$
		Theorem \ref{3.3} yields that there exists $Q \in  Hom_{\mathcal{A}}^{\ast}\left(\mathcal{X}, \mathcal{H}_{\mathcal{A}}\right)$ such that $K=T^{*} Q$.

		$(3)\Rightarrow(1)$
		For every $\xi \in \mathcal{X}$, we have
		$$
		Q \xi=\sum_{i \in I}\left\langle 	Q \xi, e_{i}\right\rangle e_{i}
		$$
		Thus
		$$
		T^{*} 	Q \xi=\sum_{i \in I}\left\langle Q \xi, e_{i}\right\rangle T^{*} e_{i}, \quad \xi \in \mathcal{X}
		$$
		Let $c_{i}=\left\langle Q \xi, e_{i}\right\rangle$, so for all $\xi \in \mathcal{X}$ we get
		$$
		\sum_{i \in I} c_{i} c_{i}^{*}=\sum_{i \in I}\left\langle Q \xi, e_{i}\right\rangle\left\langle e_{i}, Q \xi\right\rangle=\langle Q \xi, Q \xi\rangle \leqslant\|Q\|_{\infty}^{2}\langle \xi, \xi\rangle
		$$
		Since $\left\{\xi_{i}\right\}_{i \in I}$ is a Bessel sequence for $\mathcal{X}$, we get that $\left\{\xi_{i}\right\}_{i \in I}$ is an atomic system for $K$.

	\end{proof}
	\begin{corollary}
		Let $\left\{\xi_{i}\right\}_{i \in I}$ be a frame of multiplier for $\mathcal{X}$ with bounds $C, D>0$ and $K \in Hom_{\mathcal{A}}^{\ast}(\mathcal{X})$. Then $\left\{\xi_{i}\right\}_{i \in I}$ is an atomic system for $K$ with bounds $\frac{1}{C^{-1}\|K\|^{2}}$ and $D$.
	\end{corollary}

	\begin{proof}
		Let $S$ be the frame operator of $\left\{\xi_{i}\right\}_{i \in I}$.We show that condition $(2)$ of the Theorem \ref{thm3.5}  is verified. Since $\left\{S^{-1} \xi_{i}\right\}_{i \in I}$ is a frame for $\mathcal{X}$ with bounds $ D^{-1}, C^{-1}>0$
		and  $x=\sum_{n \in J}\left\langle x, f_{n}\right\rangle S^{-1} f_{n}$ for all  $x \in \mathcal{H}$, we get
		
		$$
		\left(\bar{p}_{\mathcal{X}}\left(K^{*} \xi\right)\right)^{2}=\left(\sup \left\{p(\langle\xi, K y\rangle): \bar{p}_{\mathcal{X}}(y) \leq 1\right\}\right)^{2}
		$$
		And 
		$K y=\sum_{i} m_{i} \xi_{i}$.
		
		$$
		\begin{aligned}
			\bar{p}_{\mathcal{X}}(K^{*} x)^{2} &=\left(\sup \left\{p(\langle\xi, K y\rangle): \bar{p}_{\mathcal{X}}(y) \leq 1\right\}\right)^{2} \\
			&=(\sup \{p\left\langle(\sum_{n \in J}\langle x, f_{n}\rangle K^{*} S^{-1} f_{n}, y) \right\rangle: \bar{p}_{\mathcal{X}}(y) \leq 1\})^{2} \\
			&= (\sup \{p(\sum_{n \in J}\left\langle x, f_{n}\right\rangle\left\langle K^{*} S^{-1} f_{n}, y\right\rangle: \bar{p}_{\mathcal{X}}(y) \leq 1\})^{2} \\
			&\leq \sup _{\bar{p}_{\mathcal{X}}(y) \leq 1}\left\{p\left(\sum_{n \in J}\left\langle K y, S^{-1} f_{n}\right\rangle\left\langle S^{-1} f_{n}, K y\right\rangle\right) p\left(\sum_{i \in I}\left\langle \xi, \xi_{i}\right\rangle_{M(\mathcal{X})}\left\langle \xi_{i}, \xi\right\rangle_{M(\mathcal{X})}\right)\right\}\\ 
			& \leq \sup _{\bar{p}_{\mathcal{X}}(y) \leq 1} C^{-1} p\left(\sum_{i \in I}\left\langle \xi, \xi_{i}\right\rangle_{M(\mathcal{X})}\left\langle \xi_{i}, \xi\right\rangle_{M(\mathcal{X})}\right){\bar{p}_{\mathcal{X}}(K y)}^{2} \\
			&\leq C^{-1}{\bar{p}_{\mathcal{X}}(K)^{2}}p\left(\sum_{i \in I}\left\langle \xi, \xi_{i}\right\rangle_{M(\mathcal{X})}\left\langle \xi_{i}, \xi\right\rangle_{M(\mathcal{X})}\right)
		\end{aligned}
		$$
		Hence $\left\{\xi_{i}\right\}_{i \in I}$ is an atomic system for $K$.

	\end{proof}
	\begin{theorem}
		Let $K_{1}, K_{2} \in Hom_{\mathcal{A}}^{\ast}(\mathcal{X})$. If $\left\{\xi_{i}\right\}_{i \in I}$ is an atomic system for $K_{1}$ and $K_{2}$, and $\alpha, \beta$ are scalars, then $\left\{\xi_{i}\right\}_{i=1}^{\infty}$ is an atomic system for $\lambda K_{1}+\gamma K_{2}$ and $K_{1} K_{2}$.
	\end{theorem}
	\begin{proof}
		$\{\xi\}_{i=1}^{\infty}$ is an atomic system for 
		$K_{1}$ and $K_{2}$, then there are two positives constantes $0< \lambda_{n} \leq \gamma_{n} < \infty$ $(n=1,2)$ such that 
		\begin{equation}\label{ine3.1}
			\lambda_{n}\left(\bar{p}_{\mathcal{X}}\left(K_{n}^{*} x\right)\right)^{2} \leq p \left(\sum_{i \in l}\left\langle x, \xi_{i}\right\rangle_{M(\mathcal{X})}\left\langle \xi_{i}, x\right\rangle_{M(\mathcal{X})}\right) 
			\leq \gamma_{n} \left(\bar{p}_{\mathcal{X}}(x)\right)^{2}, \text{for all} \; x \in \mathcal{X} . 
		\end{equation}
		We have
		\begin{align*}
			\left(\bar{p}_{\mathcal{X}}\left(\alpha K_{1}+K_{2} \right)^{*}x \right)^{2}&= p(\langle\left(\alpha K_{1}+K_{2} \right)^{*}x,\left(\alpha K_{1}+K_{2} \right)^{*}x\rangle) \\
			&=p(\langle\left(\alpha K_{1} \right)^{*}x,\left(\alpha K_{1} \right)^{*}x\rangle+\langle\left(\beta K_{2} \right)^{*}x,\left(\beta K_{2} \right)^{*}x\rangle)\\
			&\leq p(\alpha^{2}) C p\left(\sum_{i \in I}\left\langle x, \xi_{i}\right\rangle_{M(\mathcal{X})}\left\langle \xi_{i}, x\right\rangle_{M(\mathcal{X})}\right)\\&+ p(\beta^{2}) D p\left(\sum_{i \in I}\left\langle x, \xi_{i}\right\rangle_{M(\mathcal{X})}\left\langle \xi_{i}, x\right\rangle_{M(\mathcal{X})}\right)\\
			&\leq \max(p(\alpha^{2}) C,p(\beta^{2}) D ) p\left(\sum_{i \in I}\left\langle x, \xi_{i}\right\rangle_{M(\mathcal{X})}\left\langle \xi_{i}, x\right\rangle_{M(\mathcal{X})}\right)
		\end{align*}
		by setting $M=\max(p(\alpha^{2}) C,p(\beta^{2}) D )$
		It follows that
		$$M^{-1}\left(\bar{p}_{\mathcal{X}}\left(\alpha K_{1}+\beta K_{2} \right)^{*}x \right)^{2}\leq p\left(\sum_{i \in I}\left\langle x, \xi_{i}\right\rangle_{M(\mathcal{X})}\left\langle \xi_{i}, x\right\rangle_{M(\mathcal{X})}\right)$$
		Hence $\left\{\xi_{i}\right\}_{i=1}^{\infty}$ satisfies the lower frame condition. And from inequalities \ref{ine3.1} , we get
		$$p \left(\sum_{i \in l}\left\langle x, \xi_{i}\right\rangle_{M(\mathcal{X})}\left\langle \xi_{i}, x\right\rangle_{M(\mathcal{X})}\right) 
		\leq\frac{ \gamma_{1}+\gamma_{2}}{2} \left(\bar{p}_{\mathcal{X}}(x)\right)^{2}, \text{for all} \; x \in \mathcal{X}$$
		Hence $\{ \xi_{i} \}_{i=0}^{\infty}$ is an atomic system for $\alpha K_{1}+\beta K_{2}$.

	\end{proof}
	\section{$K$-frames in Hilbert pro-$C^{\ast}$-module}\label{sec4}
	Now we define $K$-frame in Hilbert pro-$C^{\ast}$-modules and we show that under some conditions every ordinary frame of multiplier is a $K$-frame. Then, we use Douglas' factorization theorem to demonstrat some properties of $K$-frames. We finally show the relationship between Bessel sequences and $K$-frame in pro-$C^{\ast}$-modules.
	\begin{definition}
		Let $K \in Hom_{\mathcal{A}}^{\ast}(\mathcal{X})$. $\{ \xi_{i} \}_{i=0}^{\infty}$ is a sequence in $\mathcal{X}$, $\{ \xi_{i} \}_{i=0}^{\infty}$ is called $K$-frame for $\mathcal{X}$ if there existe two positive constants $A$ and $B$ such that 
		$$A \langle K^{\ast} \xi, K^{\ast}\xi \rangle_{\mathcal{X}} \leq \sum_{i \in I}\left\langle \xi, \xi_{i}\right\rangle_{M(\mathcal{X})}\left\langle \xi_{i}, \xi \right\rangle_{M(\mathcal{X})} \leq B \langle  \xi, \xi \rangle_{\mathcal{X}} \; \; \forall \xi \in \mathcal{X} $$
		$A$ and $B$ are respectively called lower and upper bounds for $K$-frame $\{ \xi_{i} \}_{i \in I}$
	\end{definition}
	\begin{definition}
		Let $K \in Hom_{\mathcal{A}}^{\ast}(\mathcal{X})$
		be a bounded operator. A sequence $\left\{\xi_{i}\right\}_{i \in I}$ in $\mathcal{X}$ is said to be a tight $K$-frame with bound $A$ if
		$$
		A\langle K^{\ast} \xi, K^{\ast}\xi \rangle_{\mathcal{X}}=\sum_{i \in I}\left\langle \xi, \xi_{i}\right\rangle_{M(\mathcal{X})}\left\langle \xi_{i}, \xi \right\rangle_{M(\mathcal{X})}  \text { for all } \xi \in \mathcal{X}
		$$
		When $A=1$, it is called a Parseval $K$-frame $.$
	\end{definition}
	\begin{theorem}
		Let $K$ be an invertible element in $Hom_{\mathcal{A}}^{\ast}(\mathcal{X})$ such that both are uniformly bounded with $\|K\|_{\infty}^{2} $. Then every ordinary frame of multiplier is $K$-frame for $\mathcal{X}$. 
	\end{theorem}
	\begin{proof}
		Suppose that $\{ \xi_{i} \}_{i=0}^{\infty}$ is frame for $\mathcal{X}$ then there two constants $A$ and $B$ such that 
		\begin{equation}\label{7.1}
			A \langle \xi, \xi \rangle_{\mathcal{X}} \leq \sum_{i \in I}\left\langle \xi, \xi_{i}\right\rangle_{M(\mathcal{X})}\left\langle \xi_{i}, \xi\right\rangle_{M(\mathcal{X})} \leq B \langle  \xi, \xi \rangle_{\mathcal{X}} \; \; \forall \xi \in \mathcal{X} 	
		\end{equation}
		For $K $ an invertible element in $Hom_{\mathcal{A}}^{\ast}(\mathcal{X})$, we have $\langle K^{\ast} \xi, K^{\ast}\xi \rangle_{\mathcal{X}} \leq \| K\|_{\infty}^{2} \langle \xi, \xi \rangle_{\mathcal{X}}$. Then $\frac{1}{\| K\|_{\infty}^{2}} \langle K^{\ast} \xi, K^{\ast}\xi \rangle_{\mathcal{X}} \leq \langle  \xi,\xi \rangle_{\mathcal{X}}$, for all $\xi \in \mathcal{X}$. From the inequality \ref{7.1}, we have 
		$$A \frac{1}{\| K\|_{\infty}^{2}}\langle K^{\ast} \xi, K^{\ast}\xi \rangle_{\mathcal{X}} \leq A \langle  \xi,\xi \rangle_{\mathcal{X}}\leq \sum_{i \in I}\left\langle \xi, \xi_{i}\right\rangle_{M(\mathcal{X})}\left\langle \xi_{i}, \xi\right\rangle_{M(\mathcal{X})} \leq B \langle  \xi, \xi \rangle_{\mathcal{X}}$$
		Hence $\{ \xi_{i} \}_{i=0}^{\infty}$ is $K$-frame for $\mathcal{X}$.

	\end{proof}
	\begin{proposition}
		Let $\left\{\xi_{i}\right\}_{i=1}^{\infty}$ be a $K$-frame for $\mathcal{X} .$ Let $L \in \mathcal{X}$ be a bounded uniformly operator with $Ran(L) \subseteq$ $Ran(K)$. Then $\left\{f_{i}\right\}_{i=1}^{\infty}$ is a $L$-frame for $\mathcal{X}$.
	\end{proposition}
	\begin{proof}
		Let $\left\{\xi_{i}\right\}_{i=1}^{\infty}$ be  $K$-frame for $\mathcal{X}$. Then there are positive constants $A$ and $B$ such that
		\begin{equation}\label{5.2}
			A \langle K^{\ast} \xi, K^{\ast}\xi \rangle_{\mathcal{X}} \leq \sum_{i \in I}\left\langle \xi, \xi_{i}\right\rangle_{M(\mathcal{X})}\left\langle \xi_{i}, \xi\right\rangle_{M(\mathcal{X})} \leq B \langle  \xi, \xi \rangle_{\mathcal{X}} \; \; \forall \xi \in \mathcal{X}
		\end{equation} 
		Since $Ran(L) \subseteq Ran(K)$, by Douglas' theorem \ref{3.3}, there exists $\alpha>0$ such that $L L^{*} \leq \alpha^{2} K K^{*}$.
		From the inequality (\ref{5.2}), we have
		\begin{align*}
			\frac{A}{\alpha^{2}} \langle L^{\ast} \xi, L^{\ast}\xi \rangle_{\mathcal{X}}  \leq A \langle K^{\ast} \xi, K^{\ast}\xi \rangle_{\mathcal{X}} &\leq \sum_{i \in I}\left\langle x, \xi_{i}\right\rangle_{M(\mathcal{X})}\left\langle \xi_{i}, x\right\rangle_{M(\mathcal{X})}\\& \leq B \langle  \xi, \xi \rangle_{\mathcal{X}} \; \; \forall \xi \in \mathcal{X} .
		\end{align*}
		Therefore $\left\{\xi_{i}\right\}_{i=1}^{\infty}$ is a $L$-frame for $\mathcal{X}$.

	\end{proof}
	\begin{theorem}\label{thm4.5}
		Let $K \in Hom_{\mathcal{A}}^{\ast}(\mathcal{X})$ and  $\left\{\xi_{i}\right\}_{i \in I}$ be a Bessel sequence for $\mathcal{X}$ . Suppose that $T \in Hom_{\mathcal{A}}^{\ast}\left(\mathcal{X}, \mathcal{H}_{\mathcal{A}}\right)$ is given by $T(\xi)=\left\{\left\langle \xi, \xi_{i}\right\rangle\right\}_{i \in I}$ and $\overline{Ran(T)}$ is orthogonally complemented. Then $\left\{\xi_{i}\right\}_{i \in I}$ is a $K$-frame for $\mathcal{X}$ if and only if there exists a linear bounded operator $L: \mathcal{H}_{\mathcal{A}} \rightarrow \mathcal{X}$ such that $L e_{i}=\xi_{i}$ and $Ran(K) \subseteq Ran(L)$, where $\left\{e_{i}\right\}_{i}$ is the orthonormal basis for $\mathcal{H}_{\mathcal{A}}
		$.
	\end{theorem}
	\begin{proof}
		Suppose that $\left\{\xi_{i}\right\}_{i \in I}$ is $K$-frame. Then $ A \bar{p}_{\mathcal{X}}( K^{*} \xi )^{2} \leqslant\bar{p}_{\mathcal{X}}(T\xi)^{2}$ for all $\xi \in \mathcal{X}$. By Theorem \ref{th3.8}, there is $\alpha>0$ such that
		$$K K^{*} \leqslant \alpha T^{*} T .$$
		Setting $T^{*}=L$, we get $K K^{*} \leqslant \alpha L L^{*}$ and therefore $Ran(K) \subseteq Ran(L)$.
		Conversely, since $Ran(K) \subseteq Ran(L)$, by Theorem \ref{3.3} there exists $\alpha>0$ such that $K K^{*} \leqslant \alpha L L^{*}$. Therefore
		$$
		\frac{1}{\alpha}\left\langle K^{*} \xi, K^{*} \xi\right\rangle \leqslant\left\langle L^{*}\xi, L^{*} \xi\right\rangle=\sum_{i \in I}\left\langle \xi, \xi_{i}\right\rangle\left\langle \xi_{i}, \xi\right\rangle, \quad \xi \in \mathcal{X}
		$$
		Then $\left\{\xi_{i}\right\}_{i \in I}$ is $K$-frame for $\mathcal{X}$.

	\end{proof}
	\begin{proposition}\label{prop5.5}
		Let $\left\{\xi_{i}\right\}_{i \in I}$ be a  Bessel sequence of $\mathcal{X}$
		,  $\left\{\xi_{i}\right\}_{i \in I}$ is a $K$-frame with bounds $A, B>0$ if and only if $L \geqslant A K K^{*}$, where $L$ is the frame operator for $\left\{\xi_{i}\right\}_{i \in I}.$
	\end{proposition}
	\begin{proof}
		A sequence  $\left\{\xi_{i}\right\}_{i \in I}$ is a $K$-frame for $\mathcal{X}$ if and only if
		$$
		\left\langle A K K^{*} \xi, \xi\right\rangle=A\left\langle K^{*} \xi, K^{*}\xi\right\rangle \leqslant \sum_{i \in I}\left\langle \xi, \xi_{i}\right\rangle_{M(\mathcal{X})}\left\langle \xi_{i}, \xi \right\rangle_{M(\mathcal{X})}=\langle L \xi, \xi\rangle \leqslant B\langle \xi, \xi\rangle,
		$$
		Then   $\left\{\xi_{i}\right\}_{i \in I}$ is a $K$-frame  if and only if $L \geqslant A K K^{*}$.

	\end{proof}
	\begin{theorem}
		Let $\left\{\xi_{i}\right\}_{i\in I}$ be a Bessel sequence in $\mathcal{X}$. Then $\left\{\xi_{i}\right\}_{i\in I}$ is a $K$-frame for $\mathcal{X}$ if and only if $K=L^{1 / 2} L$, for some $U \in Hom_{\mathcal{A}}^{\ast}(\mathcal{X})$.
	\end{theorem}
	\begin{proof}
		Suppose $\left\{\xi_{i}\right\}_{i\in I}$ is a $K$-frame, by Proposition \ref{prop5.5}  there exist two positives constantes $A$ and $B$ 
		$$
		A K K^{*} \leq L^{1 / 2} L^{1 / 2^{*}}
		$$
		Therefore  by Douglas' theorem \ref{3.3}, $K=L^{1 / 2} U$, for some $L $ bounded in $ Hom_{\mathcal{A}}^{\ast}(\mathcal{X}) .$
		
		Conversly, let $K=L^{1 / 2} U$, for some $L$ bounded in $ Hom_{\mathcal{A}}^{\ast}(\mathcal{X}) .$  Then by Douglas' factorization theorem, $L^{1 / 2}$ majorizes $K^{*}$. Then there is a positive number $A$ such that
		
		$L \geq A^{2} K K^{*}$. Therefore by Proposition \ref{prop5.5} $\left\{\xi_{i}\right\}_{i\in I}$  is a $K$-frame for $\mathcal{X}$.

	\end{proof}
	\begin{example}
		Suppose that $\left\{u_{i}\right\}_{i=1}^{\infty}$ is an orthonormal basis in $\mathcal{H}_{\mathcal{A}}$. Define operators $L$ and $K$ on $\mathcal{H}_{\mathcal{A}}$ by $L u_{i}=u_{i-1}$ for $i>1$ and $L u_{1}=0$ and $K u_{i}=u_{i+1}$ respectively. It is clear that $\left\{K u_{i}\right\}_{i=1}^{\infty}$ is a $K$-frame for $\mathcal{H}_{\mathcal{A}}$. Suppose $\left\{K u_{i}\right\}_{i \in I}$ is a $L$-frame. Then by Proposition \ref{prop5.5}, there exists $A>0$ such that $K K^{*} \geq A L L^{*}$. Hence by Douglas, theorem, $Ran(L) \subseteq Ran(K)$. But this is contradiction to $Ran(L) \nsubseteq Ran(K)$, since $u_{1} \in Ran(L)$ but $u_{1} \notin Ran(K)$.
	\end{example}
	\begin{theorem}
		Let $K \in Hom_{\mathcal{A}}^{\ast}(\mathcal{X})$ be an uniformly bounded operator such that $Ran(K)$ is closed. The frame operator of a $K$-frame is invertible on the subspace $Ran(K)$ of $\mathcal{X}$.
	\end{theorem}
	\begin{proof}
		Suppose $\left\{\xi_{i}\right\}_{i \in I}$ is a $K$-frame for $\mathcal{X}$. Then there exists $A>0$ such that
		
		\begin{equation}\label{eq4.9}
			A \langle K^{\ast} \xi, K^{\ast}\xi \rangle_{\mathcal{X}} \leq \sum_{i \in I}\left\langle \xi, \xi_{i}\right\rangle_{M(\mathcal{X})}\left\langle \xi_{i}, \xi \right\rangle_{M(\mathcal{X})}
		\end{equation}
		Since $Ran(K)$ is closed, then $K K^{\dagger} \xi=\xi$, for all $\xi \in Ran(K)$. That is,
		$$
		\left.K K^{\dagger}\right|_{Ran(K)}=I_{Ran(K)}
		$$
		we have $I_{Ran(K)}^{*}=\left(\left.K^{\dagger}\right|_{Ran(K)}\right)^{*} K^{*}$.
		For any $\xi \in Ran(K)$, we obtain
		$$
		\langle \xi,\xi\rangle =
		\langle \left(\left.K^{\dagger}\right|_{R(K)}\right)^{*} K^{*}  \xi, \left(\left.K^{\dagger}\right|_{R(K)}\right)^{*} K^{*}  \xi \rangle \leq\left\|K^{\dagger}\right\|_{\infty}^{2} \cdot \langle K^{*} \xi,K^{*}\xi\rangle 
		$$
		Therefore
		$$ \langle K^{*} \xi,K^{*}\xi\rangle \geq  \left\|K^{\dagger}\right\|_{\infty}^{-2} \cdot\langle \xi,\xi\rangle  $$
		In combination with \ref{eq4.9}, we obtain 
		$$ \sum_{i \in I}\left\langle \xi, \xi_{i}\right\rangle_{M(\mathcal{X})}\left\langle \xi_{i}, \xi \right\rangle_{M(\mathcal{X})} \geq  A \langle K^{\ast} \xi, K^{\ast}\xi \rangle_{\mathcal{X}} \geq A \left\|K^{\dagger}\right\|_{\infty}^{-2} \cdot\langle \xi,\xi\rangle  \; \; \text{for all} \; \xi \in Ran(K). $$
		Hence, by the definition of $K$-frame, we get
		$$ A \left\|K^{\dagger}\right\|_{\infty}^{-2} \cdot\langle \xi,\xi\rangle_{\mathcal{X}} \leq \sum_{i \in I}\left\langle \xi, \xi_{i}\right\rangle_{M(\mathcal{X})}\left\langle \xi_{i}, \xi \right\rangle_{M(\mathcal{X})} \leq B \langle  \xi, \xi \rangle_{\mathcal{X}} \; \; \forall \xi \in Ran(K) $$
		Therefore
		$$ A \left\|K^{\dagger}\right\|_{\infty}^{-2} \cdot\langle \xi,\xi\rangle_{\mathcal{X}} \leq \langle S\xi ,S\xi \rangle_{\mathcal{X}} \leq B \langle \xi,\xi\rangle_{\mathcal{X}}$$
		And so $S: Ran(K) \rightarrow Ran(S)$ is a bounded linear operator and invertible on $Ran(K)$.

	\end{proof}
	\begin{theorem}
		Let $K \in Hom_{\mathcal{A}}^{\ast}(\mathcal{X})$ be uniformly bounded such that $Ran(K)$ is dense. Let $\left\{\xi_{i}\right\}_{i \in I}$ be a $K$-frame and $T \in Hom_{\mathcal{A}}^{\ast}(\mathcal{X})$ be uniformly bounded such that $Ran(T)$ is closed. If $\left\{T \xi_{i}\right\}_{i \in I}$ is a $K$-frame for $\mathcal{X}$, then $T$ is surjective.
	\end{theorem}
	\begin{proof}
		Let's assume that $\left\{T \xi_{i}\right\}_{i\in I}$ is a $K$-frame for $\mathcal{X}$ with frame bounds $A$ and $B$. Then for any $\xi \in \mathcal{X}$,
		\begin{equation}\label{eq5.7}
			A \langle K^{*}\xi, K^{*} \xi \rangle \leq \sum_{i \in I} \langle \xi, T \xi_{i}\rangle_{M(\mathcal{X})} \langle T \xi_{i} ,\xi \rangle_{M(\mathcal{X})} \leq B \langle \xi,\xi \rangle.	 \end{equation}
		Since  $Ran(K)$ is dense, $\mathcal{X}=\overline{Ran(K)}$, so $K^{*}$ is injective. Then from \ref{eq5.7}, $T^{*}$ is injective since $Ker\left(T^{*}\right) \subseteq Ker\left(K^{*}\right)$. Moreover, $Ran(T)=Ker\left(T^{*}\right)^{\perp}=\mathcal{X}$. Thus $T$ is surjective.

	\end{proof}
	\begin{theorem}
		Let $K \in Hom_{\mathcal{A}}^{\ast}(\mathcal{X})$ be uniformly bounded and let $\left\{\xi_{i}\right\}_{i \in I}$ be a $K$-frame for $\mathcal{X}$. If $T \in Hom_{\mathcal{A}}^{\ast}(\mathcal{X})$ is uniformly bounded and has a closed range with $T K=K T$, then $\left\{T \xi_{i}\right\}_{i\in I}$ is a $K$-frame for $Ran(T)$.
	\end{theorem}
	\begin{proof}
		As $T$ has a closed range, it has the generalized inverse $T^{\dagger}$ such that $T T^{\dagger}=I$.
		Now $I=I^{*}=T^{\dagger^{*}} T^{*}$. Then for each $\xi \in Ran(T), K^{*} \xi=T^{\dagger^{*}} T^{*} K^{*} \xi$, so we have
		$$
		\langle K^{*}\xi,K^{*}\xi \rangle = \langle T^{\dagger^{*}} T^{*} K^{*} \xi, T^{\dagger^{*}} T^{*} K^{*} \xi \rangle \leq\left\|T^{\dagger^{*}}\right\|_{\infty}^{2} \langle T^{*} K^{*} \xi, T^{*} K^{*} \xi \rangle
		$$
		Therefore $$\left\|T^{\dagger^{*}}\right\|_{\infty}^{-2}\langle K^{*}\xi,K^{*}\xi \rangle \leq \langle T^{*} K^{*} \xi, T^{*} K^{*} \xi \rangle.$$
		Now for each $\xi \in Ran(T)$,
		$$
		\begin{aligned}
			\sum_{i \in I} \langle \xi, T \xi_{i}\rangle_{\mathcal{X}} \langle T \xi_{i} ,\xi \rangle_{\mathcal{X}} &=\sum_{i \in I} \langle T^{\ast}\xi,  \xi_{i}\rangle_{\mathcal{X}} \langle \xi_{i} ,T^{\ast}\xi \rangle_{\mathcal{X}} \geq A \langle K^{*} T^{*} \xi, K^{*} T^{*} \xi \rangle \\
			&= A \langle  T^{*} K^{*} \xi, T^{*} K^{*}  \xi \rangle \\
			& \geq A \left\|T^{\dagger^{*}}\right\|_{\infty}^{-2} \langle K^{*} \xi, K^{*} \xi \rangle.
		\end{aligned}
		$$
		Since $\left\{\xi_{i}\right\}_{i \in I}$ is a Bessel sequence with bound $B$, for each $\xi \in Ran(T)$, we have
		$$
		\sum_{i \in I} \langle \xi, T \xi_{i}\rangle_{M(\mathcal{X})} \langle T \xi_{i} ,\xi \rangle_{M(\mathcal{X})} =\sum_{i \in I} \langle T^{\ast}\xi,  \xi_{i}\rangle_{M(\mathcal{X})}\langle \xi_{i} ,T^{\ast}\xi \rangle_{M(\mathcal{X})} \leq B \langle T^{*} \xi, T^{*} \xi \rangle \leq B \|T\|_{\infty}^{2} \langle \xi, \xi \rangle.
		$$
		Therefore $\left\{T \xi_{i}\right\}_{i \in I}$ is a $K$-frame for $Ran(T)$.

	\end{proof}
	
	\begin{theorem}
		Let $K \in Hom_{\mathcal{A}}^{\ast}(\mathcal{X})$ be uniformly bounded such that $Ran(K)$ is dense. Let $\left\{\xi_{i}\right\}_{i\in I}$ be a $K$-frame and let $T \in Hom_{\mathcal{A}}^{\ast}(\mathcal{X})$ be uniformly bounded such that $Ran(T)$ is closed. If $\left\{T \xi_{i}\right\}_{i \in I}$ and $\left\{T^{*} \xi_{i}\right\}_{i \in I}$ are $K$-frames for $\mathcal{X}$, then $T$ is invertible.
	\end{theorem}
	\begin{proof}
		Suppose that $\left\{T \xi_{i}\right\}_{i \in I}$ is a $K$-frame for $\mathcal{X}$ with frame bounds $A_{1}$ and $B_{1}$. Then for any $\xi \in \mathcal{X}$,
		\begin{equation}\label{eq 5.9}
			A_{1} \langle K^{*} \xi , K^{*}\xi \rangle \leq 	\sum_{i \in I} \langle \xi, T \xi_{i}\rangle_{M(\mathcal{X})} \langle T \xi_{i} ,\xi \rangle_{M(\mathcal{X})} \leq B_{1}\langle \xi , \xi \rangle
		\end{equation}
		As $Ran(K)$ is dense, $K^{*}$ is injective. Then from  \ref{eq 5.9}, $T^{*}$ is injective since $Ker\left(T^{*}\right) \subseteq Ker\left(K^{*}\right)$. Moreover $Ran(T)=Ker\left(T^{*}\right)^{\perp}=\mathcal{X}$. Then $T$ is surjective.
		
		Suppose $\left\{T^{*} f_{i}\right\}_{i=1}^{\infty}$ is a $K$-frame for $\mathcal{X}$ with frame bounds $A_{2}$ and $B_{2}$. Then for any $\xi \in \mathcal{X}$,
		\begin{equation}\label{eq 5.9.2}
			A_{2} \langle K^{*} \xi , K^{*}\xi \rangle \leq 	\sum_{i \in I} \langle \xi, T^{\ast} \xi_{i}\rangle_{M(\mathcal{X})} \langle T^{\ast} \xi_{i} ,\xi \rangle_{M(\mathcal{X})} \leq B_{2}\langle \xi , \xi \rangle
		\end{equation}
		As $K$ has a dense range, $K^{*}$ is injective. Then from \ref{eq 5.9.2}, $T$ is injective since $Ker(T) \subseteq Ker\left(K^{*}\right)$. Therefore $T$ is bijective. Using the Bounded Inverse Theorem, $T$ is invertible.

	\end{proof}
	\begin{theorem}
		Let $K \in Hom_\mathcal{A}^{\ast}(\mathcal{X})$ and let $\left\{\xi_{i}\right\}_{i \in I}$ be a $K$-frame for $\mathcal{X}$ and let $T \in Hom_\mathcal{A}^{\ast}(\mathcal{X})$ be uniformly bounded and be a co-isometry with $T K=K T$. Then $\left\{T \xi_{i}\right\}_{i=1}^{\infty}$ is a $K$-frame for $\mathcal{X}$.
	\end{theorem}
	\begin{proof}
		Suppose $\left\{\xi_{i}\right\}_{i \in I}$ is a $K$-frame for $\mathcal{X}$. For every $\xi \in \mathcal{X}$
		$$
		\begin{aligned}
			\sum_{i \in I} \langle \xi, T \xi_{i}\rangle_{\mathcal{X}} \langle T \xi_{i} ,\xi \rangle_{\mathcal{X}}  &=	\sum_{i \in I} \langle T^{\ast} \xi,  \xi_{i}\rangle_{\mathcal{X}} \langle  \xi_{i} ,T^{\ast}\xi \rangle_{\mathcal{X}}  \geq A\langle K^{*} T^{*} \xi, K^{*} T^{*} \xi \rangle \\
			&=A\langle T^{*} K^{*} \xi, T^{*} K^{*} \xi \rangle \\
			&=A\langle  K^{*} \xi,  K^{*} \xi \rangle 
		\end{aligned}
		$$
		It is obvious that $\left\{T \xi_{i}\right\}_{i \in I}$ is a Bessel sequence. Since $\left\{\xi_{i}\right\}_{i\in I}$ is a Bessel sequence, for each $\xi \in \mathcal{X}$ we have
		$$
		\sum_{i \in I} \langle \xi, T \xi_{i}\rangle_{\mathcal{X}} \langle T \xi_{i} ,\xi \rangle_{\mathcal{X}} 
		=\sum_{i \in I} \langle T^{\ast} \xi,  \xi_{i}\rangle_{\mathcal{X}} \langle  \xi_{i} ,T^{\ast} \xi \rangle_{\mathcal{X}} \leq B \|T\|_{\infty}^{2} \langle \xi,\xi \rangle.
		$$
		Therefore $\left\{T \xi_{i}\right\}_{i \in I}$ is a K-frame for $\mathcal{X}$.

	\end{proof}
	
	\section{Sums of $K$-Frames}\label{sec5}
	
	In the following section we show that the sums of $K$-frames under some conditions is again a $K$-frame in Hilbet pro-$C^{\ast}$-modules.
	\begin{theorem}
		Let $\left\{\xi_{i}\right\}_{i \in I}$ and $\left\{\eta_{i}\right\}_{i \in I}$  be two $K$-frames for $\mathcal{X}$ and  let $L_{1}$ and $L_{2}$ be   respectively their  corresponding operators in Proposition \ref{prop5.5}. If $L_{1} L_{2}^{*}$ and $L_{2} L_{1}^{*}$ are positive operators and $Ran\left(L_{1}\right)+Ran\left(L_{2}\right)$ is closed, then $\left\{\xi_{i}+\eta_{i}\right\}_{i \in I}$ is a $K$-frame for $\mathcal{X}$.
	\end{theorem}
	\begin{proof}
		Suppose that  $\left\{\xi_{i}\right\}_{i \in I}$ and $\left\{\eta_{i}\right\}_{i \in I}$  are two $K$-frames for $\mathcal{X}$. By the assumption, there are two bounded operators $L_{1}$ and $L_{2}$ such that $L_{1} u_{i}=\xi_{i},$ $L_{2} u_{i}=\eta_{i}$ and $Ran(K) \subseteq Ran\left(L_{1}\right), Ran(K) \subseteq Ran\left(L_{2}\right)$, we denote by $\left\{u_{i}\right\}_{i  \in I}$ is an orthonormal basis for $\mathcal{H}_{\mathcal{A}}$
		
		So $Ran(K) \subseteq Ran\left(L_{1}\right)+Ran\left(L_{2}\right)$, by Corollary \ref{crll3.5} 
		$$
		K K^{*} \leq \alpha^{2}\left(L_{1} L_{1}^{*}+L_{2} L_{2}^{*}\right), \text { for some } \alpha>0
		$$
		For every $\xi \in \mathcal{X}$, we have
		$$\begin{aligned} \sum_{i\in I}\langle \xi, \xi_{i}+\eta_{i}\rangle\langle \xi_{i}+\eta_{i}, \xi\rangle &=\sum_{i \in I}\langle (L_{1}^{*}+L_{2}^{*}) \xi, u_{i}\rangle \langle u_{i},( L_{1}{ }^{*}+L_{2}^{*}) \xi\rangle \\ &=\sum_{i \in I}\langle(L_{1}+L_{2})^{*} \xi, u_{i}\rangle \langle u_{i}, (L_{1}+L_{2})^{*} \xi \rangle  \\ &= \langle (L_{1}+L_{2})^{*} \xi,(L_{1}+L_{2})^{*} \xi \rangle \\ &= \langle L_{1}{ }^{*} \xi, L_{1}{ }^{*} \xi\rangle+\langle L_{1}{ }^{*} \xi, L_{2}{ }^{*} \xi\rangle \\ &+ \langle L_{2}{ }^{*} \xi, L_{1}{ }^{*} \xi\rangle+ \langle L_{2}{ }^{*} \xi, L_{2}{ }^{*} \xi\rangle \\ & \geqslant\langle(L_{1} L_{1}{ }^{*}+L_{2} L_{2}{ }^{*}) \xi, \xi\rangle \\ &\geqslant \frac{1}{\alpha^{2}}\left(\left\langle K K^{*} \xi, \xi\right\rangle\right).
			\\ &\geqslant \frac{1}{\alpha^{2}}\left(\left\langle K^{*} \xi, K^{*} \xi\right\rangle .\right.\end{aligned}$$
		Therefore 	$\left\{\xi_{i}+\eta_{i}\right\}_{i \in I}$ is a $K$-frame for $\mathcal{X}$.

	\end{proof}
	\begin{corollary}
		Let $\left\{\xi_{i}\right\}_{i\in I}$ and $\left\{\eta_{i}\right\}_{i \in I}$ be $K$-frames for $\mathcal{X}$ with frame operators $L_{1}$ and $L_{2}$ respectively. Then $K=L_{1}^{1 / 2} T_{1}+L
		_{2}^{1 / 2} T_{2}$, for some bounded operators $T_{1}$ and  $T_{2}$ in $Hom_\mathcal{A}^{\ast}(\mathcal{X})$.
	\end{corollary}
	\begin{proof}
		Since $\left\{\xi_{i}\right\}_{i\in I}$ and $\left\{\eta_{i}\right\}_{i \in I}$ are $K$-frames for $\mathcal{X}$, by Proposition \ref{prop5.5}, there are positive constants $A_{1}$ and $A_{2}$ such that
		$$
		L_{1} \geq A_{1} K K^{*} \text { and } L_{2} \geq A_{2} K K^{*}
		$$
		Then by Douglas'  theorem, we have $$Ran(K) \subseteq Ran\left(L_{1}^{1 / 2}\right)$$ and $$Ran(K) \subseteq Ran\left(L_{2}^{1 / 2}\right).$$ Hence $Ran(K) \subseteq Ran\left(S_{1}^{1 / 2}\right)+Ran\left(S_{2}^{1 / 2}\right) .$ Hence by Corollary \ref{crll3.5}, there exists two bounded operators  $T_{1}, T_{2} $ in  $Hom_{\mathcal{A}}^{\ast}(\mathcal{X})$ such that $K=L_{1}^{1 / 2} T_{1}+L_{2}^{1 / 2} T_{2}$.

	\end{proof}
	\begin{theorem}
		Suppose that  $\left\{\xi_{i}\right\}_{i \in I}$ is a $K$-frame for $\mathcal{X}$ such that $L$ is its frame operator and let $T$ be a positive operator. Then $\left\{\xi_{i}+T \xi_{i}\right\}_{i \in I}$ is a $K$-frame. Furthermore $\left\{\xi_{i}+T^{n} f_{i}\right\}_{i \in I}$ is a $K$-frame for $\mathcal{X}$ for any natural number $n$.
	\end{theorem}
	\begin{proof}
		Let $\left\{\xi_{i}\right\}_{i \in I}$ be a $K$-frame for $\mathcal{X}$. Then by Proposition \ref{prop5.5}, there is $A>0$ such that $L \geq A K K^{*}$. $(I+T) L(I+ T)^{*}$ is the frame operator for $\left\{\xi_{i}+T \xi_{i}\right\}_{i \in I}$ for the raison that for each $\xi \in \mathcal{X}$,
		$$
		\begin{aligned}
			\sum_{i\in I}\left\langle \xi,\left(\xi_{i}+T \xi_{i}\right)\right\rangle\left(\xi_{i}+T \xi_{i}\right) &=(I+T) \sum_{i\in I}\left\langle \xi,(I+T) \xi_{i}\right\rangle \xi_{i} \\
			&=(I+T) L(I+T)^{*} \xi
		\end{aligned}
		$$
		In addition, we have
		$$
		(I+T) L(I+T)^{*}=L+L T^{*}+T L+T L T^{*} \geq L \geq A K K^{*}
		$$
		By Propostion \ref{prop5.5}, we can say that $\left\{\xi_{i}+T \xi_{i}\right\}_{i \in I}$ is a $K$-frame for $\mathcal{X}$.
		For every given natural number $n$, the frame operator for $\{\xi_{i}+T^{n} \xi_{i}\}_{i \in I}$ is $(I+T^{n}) L (I+ T^{n})^{*} \geq L$. Therefore $\left\{\xi_{i}+T^{n} \xi_{i}\right\}_{i=1}^{\infty}$ is a $K$-frame for $\mathcal{X}$.

	\end{proof}
	\begin{corollary}
		Let $\left\{\xi_{i}\right\}_{i \in I}$ be a $K$-frame for $\mathcal{X}$  and let $L$ be its  frame operator. Suppose that $\left\{I_{1}, I_{2}\right\}$ is a partition of $\mathbb{N}$. For $j=1,2$, let $L_{j}$ be the frame operator for the Bessel sequence $\left\{\xi_{i}\right\}_{i \in I_{j}}.$ Then
		$$
		\left\{\xi_{i}L_{1}^{m} \xi_{i}\right\}_{i \in I_{1}} \cup\left\{\xi_{i}+L_{2}^{n} \xi_{i}\right\}_{i \in I_{2}}
		$$
		is a $K$-frame for $\mathcal{X}$ for every $m, n \in \mathbb{N}$.
	\end{corollary}
	\begin{proof}
		For every  $m \in \mathbb{N}$,  $L^{m}$ can be defined as follow
		$$
		L^{m} \xi=\sum_{i\in I_{j}}\left\langle \xi, L^{\frac{m-1}{2}} \xi_{i}\right\rangle L^{\frac{m-1}{2}} \xi_{i}
		$$
		For each $\xi \in \mathcal{X}$,
		$$\begin{aligned} \sum_{i\in I_{1}}\left\langle \xi, \xi_{i}+L_{1}^{m} \xi_{i}\right\rangle\left(\xi_{i}+L_{1}^{m} \xi_{i}\right) &=\left(I+L_{1}^{m}\right) \sum_{i\in I_{1}}\left\langle \xi, \xi_{i}+L_{1}^{m} \xi_{i}\right\rangle \xi_{i} \\ &=\left(I+L_{1}^{m}\right) \sum_{i\in I_{1}}\left\langle \xi,\left(I+L_{1}^{m}\right) \xi_{i}\right\rangle \xi_{i} \\ &=\left(I+L_{1}^{m}\right) L_{1}\left(I+L_{1}^{m}\right)^{*} \xi \\ &=\left(I+L_{1}^{m}\right)\left(L_{1}+L_{1}^{(1+m)}\right) \xi \\ &=\left(L_{1}+2 L_{1}^{(1+m)}+L_{1}^{(1+2 m)}\right) \xi . \end{aligned}$$
		Therefore $L_{1}+2 L_{1}^{(1+m)}+$ $L_{1}^{(1+2 m)}$ and $L_{2}+2 L_{2}^{(1+n)}+L_{2}^{(1+2 n)}$ are the frame operators for $\left\{\xi_{i}+L_{1}^{m} \xi_{i}\right\}_{i \in I_{1}}$ and $\left\{\xi_{i}+L_{2}^{n} \xi_{i}\right\}_{i \in I_{2}}$ respectively. Consider that $L_{0}$ is the frame operator for $\left\{\xi_{i}+L_{1}^{m} \xi_{i}\right\}_{i \in I_{1}} \cup\left\{\xi_{i}+L_{2}^{n} \xi_{i}\right\}_{i \in I_{2}} .$ 
		
		As $\left\{\xi_{i}\right\}_{i=1}^{\infty}$ is a $K$-frame for $\mathcal{X}$, then there exists $A>0$ such that $L \geq A K K^{*}$ and $L_{0} \geq L_{1}+L_{2}=L \geq A K K^{*}$. From which $\left\{\xi_{i}+L_{1}^{m} \xi_{i}\right\}_{i \in I_{1}} \cup\left\{\xi_{i}+L_{2}^{n} \xi_{i}\right\}_{i \in I_{2}}$ is a $K$-frame for $\mathcal{X}$.

	\end{proof}
	\begin{theorem}
		Let $\left\{\xi_{i}\right\}_{i \in I}$ and $\left\{\eta_{i}\right\}_{i \in I}$ be Parseval $K$-frames for $\mathcal{X
		}$, with synthesis operators $L_{1}$ and $L_{2}$ respectively. If $L_{1} L_{2}^{*}=0$ then $\left\{\xi_{i}+\eta_{i}\right\}_{i \in I}$ is a $2$-tight $K$-frame for $\mathcal{X}$.
	\end{theorem}
	\begin{proof}
		Suppose $\left\{\xi_{i}\right\}_{i \in I}$ and $\left\{\eta_{i}\right\}_{i \in I}$ are two Parseval $K$-frames for $\mathcal{X}$. Then there are transform operators $L_{1}, L_{2} \in Hom_{\mathcal{A}}^{\ast}(\mathcal{X})$ such that $L_{1} e_{i}=\xi_{i}$ and $L_{2} e_{i}=\eta_{i}$ with $Ran(K)=Ran\left(L_{1}\right), Ran(K)=Ran\left(L_{2}\right)$ respectively. For each $\xi \in \mathcal{X}$, we have
		\begin{align*}
			\sum_{i \in I}\left\langle\xi, \xi_{i}+\eta_{i}\right\rangle_{M(\mathcal{X})}\left\langle \xi_{i}+\eta_{i}, \xi\right\rangle_{M(\mathcal{X})} 
			&=\langle \left(L_{1}+L_{2}\right)^{*} \xi, \left(L_{1}+L_{2}\right)^{*} \xi  \rangle_{\mathcal{X}} \\
			&=\langle L_{1}^{*} \xi, L_{1}^{*} \xi  \rangle_{\mathcal{X}}  +\left\langle L_{2} L_{1}^{*} \xi, \xi\right\rangle_{\mathcal{X}} \\&+\left\langle L_{1} L_{2}^{*} \xi,L_{1} L_{2}^{*} \xi\right\rangle_{\mathcal{X}} +\langle L_{2}^{*} \xi, L_{2}^{*} \xi  \rangle_{\mathcal{X}}   \\
			&=\langle L_{1}^{*} \xi, L_{1}^{*} \xi  \rangle_{\mathcal{X}} +\langle L_{2}^{*} \xi, L_{2}^{*} \xi  \rangle_{\mathcal{X}}  \\
			&=2\langle  K^{*} \xi, K^{*} \xi  \rangle_{\mathcal{X}} 
		\end{align*}
	\end{proof}
	\begin{theorem}
		Let $\left\{\xi_{i}\right\}_{i\in I}$ and $\left\{\eta_{i}\right\}_{i\in I}$ be $K$-frames for $\mathcal{X}$, and let $L_{1}$ and $L_{2}$ be synthesis operators for sequences $\left\{\xi_{i}\right\}_{i\in I}$ and $\left\{\eta_{i}\right\}_{i\in I}$ respectively, such that $L_{1} L_{2}^{*}=0$ and let $T_{j} \in Hom_{\mathcal{A}}^{\ast}(\mathcal{X})$ an operator uniformly bounded with $R\left(L_{j}\right) \subseteq R\left(T_{j} L_{j}\right)$, for $j=1,2$. Then $\left\{T_{1} \xi_{i}+T_{2} \eta_{i}\right\}_{i\in I}$ is a $K$-frame for $\mathcal{X}$.
	\end{theorem}
	\begin{proof}
		Suppose that $\left\{\xi_{i}\right\}_{i\in I}$ and $\left\{\eta_{i}\right\}_{i\in I}$ are two $K$-frames for $\mathcal{X}$. Then by Theorem \ref{thm4.5}, there exists an orthonormal basis $\left\{e_{i}\right\}_{i=1}^{\infty}$ in $\mathcal{H}_{\mathcal{A}}$ such that $L_{1} e_{i}=\xi_{i}, L_{2} e_{i}=\eta_{i}$
		and $R(K) \subseteq R\left(L_{1}\right), R(K) \subseteq R\left(L_{2}\right)$. For each $\xi \in \mathcal{X}$
		$$
		\begin{aligned}
			\sum_{i \in I}\langle \xi, T_{1} \xi_{i}+T_{2} \eta_{i}\rangle_{M(\mathcal{X})} \langle \xi, T_{1} \xi_{i}+T_{2} \eta_{i}\rangle_{M(\mathcal{X})} &=  \sum_{i \in I}\langle \xi, T_{1}  T_{1} L_{1} e_{i}+T_{2} L_{2} e_{i} \rangle_{M(\mathcal{X})} \langle \xi, T_{1} L_{1} e_{i}+T_{2} L_{2} e_{i} \rangle_{M(\mathcal{X})} \\
			&=\langle \left(T_{1} L_{1}+T_{2} L_{2}\right)^{*} \xi,  \left(T_{1} L_{1}+T_{2} L_{2}\right)^{*} \xi \rangle \\
			&=\langle \left(T_{1} L_{1}\right)^{*} \xi, \left(T_{1} L_{1}\right)^{*} \xi \rangle +\left\langle T_{2} L_{2} L_{1}^{*} T_{1}^{*} \xi, \xi\right\rangle \\
			&+\left\langle T_{1} L_{1} L_{2}^{*} T_{2}^{*} \xi,\xi\right\rangle+\langle \left(T_{2} L_{2}\right)^{*} \xi, \left(T_{2} L_{2}\right)^{*} \xi \rangle       \\
			&=\langle \left(T_{1} L_{1}\right)^{*} \xi, \left(T_{1} L_{1}\right)^{*} \xi \rangle +\langle \left(T_{2} L_{2}\right)^{*} \xi, \left(T_{2} L_{2}\right)^{*} \xi \rangle 
		\end{aligned}
		$$
		We have that $Ran(K) \subseteq Ran\left(L_{j}\right) \subseteq Ran\left(T_{j} L_{j}\right)$ for $j=1,2$. So by Douglas' factorization theorem, for each $j=1,2$, there exists $\alpha_{j}>0$ such that
		$$
		K K^{*} \leq \alpha_{j}\left(T_{j} L_{j}\right)\left(T_{j} L_{j}\right)^{*}
		$$
		
		Then from the above inequality, for each $\xi \in \mathcal{X}$
		$$
		\begin{aligned}
			\sum_{i \in I}\langle \xi, T_{1} \xi_{i}+T_{2} \eta_{i}\rangle_{M(\mathcal{X})} \langle \xi, T_{1} \xi_{i}+T_{2} \eta_{i}\rangle_{M(\mathcal{X})}&=\langle \left(T_{1} L_{1}\right)^{*} \xi, \left(T_{1} L_{1}\right)^{*} \xi \rangle +\langle \left(T_{2} L_{2}\right)^{*} \xi, \left(T_{2} L_{2}\right)^{*} \xi \rangle \\ &\geq\left(\frac{1}{\alpha_{1}}+\frac{1}{\alpha_{2}}\right) \langle K^{*} \xi,K^{*} \xi \rangle
		\end{aligned}
		$$
		Hence $\left\{T_{1} \xi_{i}+T_{2} \eta_{i}\right\}_{i\in I}$ is a $K$-frame for $\mathcal{X}$.

	\end{proof}
	\vspace{1cm}
	
	\medskip
	
	\section*{Declarations}
	
	\medskip
	
	\noindent \textbf{Availablity of data and materials}\newline
	\noindent Not applicable.
	
	\medskip

	\noindent \textbf{Competing  interest}\newline
	\noindent The authors declare that they have no competing interests.

	\medskip
	
	\noindent \textbf{Fundings}\newline
	\noindent  Authors declare that there is no funding available for this article.

	\medskip
	
	\noindent \textbf{Authors' contributions}\newline
	\noindent The authors equally conceived of the study, participated in its
	design and coordination, drafted the manuscript, participated in the
	sequence alignment, and read and approved the final manuscript. 
	
	\medskip
	
	\noindent \textbf{Acknowledgements}\newline
	\noindent The authors are thankful to the area editor and referees for giving valuable comments and suggestions
	
	\medskip 
	
	\vspace{1cm}
	
	\bibliographystyle{amsplain}

\begin{thebibliography}{99}
		
		\bibitem{Azhini} \textsc{M. Azhini and N. Haddadzadeh}, Fusion frames in Hilbert modules over pro-$C^{\ast}$-algebras,\textit{Int. J. Industrial Math} (2013),109-118.
		
		\bibitem{Douglas}
		\textsc{R. G. Douglas}, On majorization, factorization, and range inclusion of operators on Hilbert space, \textit{Proceedings of the American Mathematical Society } \textbf{17.2} (1966): 413-415.
		
		\bibitem{Duf} R. J. Duffin,  A. C. Schaeffer , Trans. Amer. Math. Soc. 72, 1952.  341-366.
		Doi: doi.org/10.1090/S0002-9947-1952-0047179-6
		\bibitem{Fang} X. Fang, M. S. Moslehian, Q. Xu ,  On majorization and range inclusion of operators on Hilbert $C^{\ast}$-modules, Linear and Multilinear Algebra. 2018.
		Doi: doi.org/10.1080/03081087.2017.1402859
		\bibitem{Fan}
		X.Fang, J. Yu , H. Yao, Solutions to operator equations on Hilbert $C^{\ast}$-modules, Linear Algebra
		Appl. 2009;431:2142–2153.
		
		
		
		
		\bibitem{Frag}
		\textsc{M. Fragoulopoulou} , An introduction to the representation theory of
		topological $\ast$-algebras,\textit{ Schriftenreihe, Univ. Münster}, \textbf{48} (1988), 1-81.
		
		\bibitem{Frago}
		\textsc{M. Fragoulopoulou}, Tensor products of enveloping locally $C^{\ast}$-algebras, \textit{ Schriftenreihe, Univ. Münster} (1997), 1-81.
		
		\bibitem{Gab} D. Gabor, \emph{Theory of communications}, Journal of the Institution of Electrical Engineers 93(26), 429–457. 1946.  Doi: doi.org/10.1049/JI-3-2.1946.0074
		
		\bibitem{Gavruta}
		\textsc{L. Găvruţa}, Frames for operators, \textit{Applied and Computational Harmonic Analysis} \textbf{32.1} (2012): 139-144.
		
		\bibitem{Inoue}\textsc{ A. Inoue}, Locally $C^{\ast}$-algebra, \textit{Memoirs of the Faculty of Science, Kyushu University. Series A, Mathematics} \textbf{25.2 }(1972): 197-235.
		
		\bibitem{Joitaa}
		\textsc{M. Joita}, Hilbert modules over locally C*-algebras, \textit{Editura Universităţii din Bucureşti}, 2006.
		
		\bibitem{Joita}
		\textsc{M. Joita}, On frames in Hilbert modules over pro-$C^{\ast}$-algebras, \textit{Topol. Appl, }\textbf{156} (2008), 83-92.
		
		\bibitem{Lance}
		\textsc{E. C. Lance}, Hilbert $C^{\ast}$-modules, \textit{London Math. Soc}, \textbf{210}.
		Univ. Press, Cambridge, 1995.
		
		\bibitem{Mallios}
		\textsc{A. Mallios}, Topological algebras: Selected Topics, North Holland, \textit{Elsevier}, 2011.
		
		\bibitem{Philip}
		\textsc{N.C. Phillips}, Inverse limits of C*-algebras, \textit{J. Operator Theory} \textbf{19} (1988) 159-195.
		
		\bibitem{Philips}
		\textsc{N. C. Phillips}, Representable K-theory for $\sigma$ -$C^{\ast}$-algebras,\textit{ K-Theory},
		\textbf{3}(1989), 441-478.
		
		\bibitem{Zhura}
		\textsc{Yu. I. Zhuraev, F. Sharipov}, Hilbert modules over locally $C^{*}$-algebras, \textit{arXiv preprint math/0011053} (2000).
		
		
		
		
	\end{thebibliography}

\end{document}